\documentclass[11pt,a4paper]{article}

\usepackage{amsfonts}
\usepackage{amsthm}
\usepackage{amsmath}
\usepackage{amssymb}
\usepackage{amscd}
\usepackage{mathrsfs}
\usepackage[mathscr]{eucal}
\usepackage{graphicx}
\usepackage{epstopdf} 
\usepackage{pict2e}
\usepackage{epic}
\usepackage{xcolor}
\usepackage{float}
\usepackage{mathtools}
\usepackage{centernot}

\usepackage{cite}
\usepackage{hyperref}
\hypersetup{colorlinks=true, urlcolor= black, linkcolor=black, citecolor=black}

\usepackage[utf8]{inputenc}
\usepackage[T1]{fontenc}
\usepackage{lmodern}
\usepackage{dsfont}
\usepackage{indentfirst}

\numberwithin{equation}{section}

\theoremstyle{plain}
\newtheorem{Thm}{Theorem}[section]
\newtheorem{Lem}[Thm]{Lemma}

\newtheorem{Prop}[Thm]{Proposition}
\newtheorem{Question}{Open problem}
\newtheorem*{Conj}{Conjecture}
\newtheorem{Claim}[Thm]{Claim}

\theoremstyle{definition}
\newtheorem{Def}[Thm]{Definition}

\newtheorem{Rem}[Thm]{Remark}

\newtheorem*{Acknowledgements}{Acknowledgements}

\usepackage[margin=3cm]{geometry}

\newcommand\n{\mathbf{n}}
\newcommand\m{\mathbf{m}}
\newcommand\sn{\partial\mathbf{n}}

\newcommand\bfc{\mathbf{c}}
\newcommand\bfC{\mathbf{C}}
\newcommand{\connect}{\xleftrightarrow}

\title{The incipient infinite cluster of the FK-Ising model in dimensions $d\geq 3$ and the susceptibility of the high-dimensional Ising model}
\begin{document}

\author{Romain Panis\footnotemark[1]\footnote{Institut Camille Jordan, \url{panis@math.univ-lyon1.fr}} }
\maketitle

\begin{abstract} We consider the critical FK-Ising measure $\phi_{\beta_c}$ on $\mathbb Z^d$ with $d\geq 3$. We construct the measure $\phi^\infty:=\lim_{|x|\rightarrow \infty}\phi_{\beta_c}[\:\cdot\: |\: 0\connect{}x]$ and prove it satisfies $\phi^\infty[0\connect{}\infty]=1$. This corresponds to the natural candidate for the \emph{incipient infinite cluster} measure of the FK-Ising model. Our proof uses a result of Lupu and Werner (Electron. Commun. Probab., 2016)
that relates the FK-Ising model to the random current representation of the Ising model, together with a \emph{mixing property} of random currents recently established by Aizenman and Duminil-Copin (Ann. Math., 2021).

We then study the susceptibility $\chi(\beta)$ of the nearest-neighbour Ising model on $\mathbb Z^d$. When $d>4$, we improve a previous result of Aizenman (Comm. Math. Phys., 1982) to obtain the existence of $A>0$ such that, for $\beta<\beta_c$,
\begin{equation*}
	\chi(\beta)= \frac{A}{1-\beta/\beta_c}(1+o(1)),
\end{equation*}
where $o(1)$ tends to $0$ as $\beta$ tends to $\beta_c$.
Additionally, we relate the constant $A$ to the incipient infinite cluster of the double random current. 
\end{abstract}

\section{Introduction}

In the context of independent bond percolation on the hypercubic lattice $\mathbb Z^d$ in dimensions $d\geq 2$, it is believed\footnote{The absence of infinite cluster at the critical point was successfully proved when $d=2$ \cite{Kesten1980criticalproba}, and when $d>10$ \cite{HaraSlade1990Perco,HaraDecayOfCorrelationsInVariousModels2008,FitznervdHofstad2017Perco-d>10}. Proving the corresponding result when $3\leq d \leq 10$ is one of the main open problems in percolation theory.} that there is no infinite cluster at the critical threshold $p_c$. However, the expected size of the cluster of the origin (or \emph{susceptibility}) is infinite (see \cite{AizenmanNewmanTreeGraphInequalities1984}). This implies that, from the perspective of an observer at the origin, increasingly large (but finite) clusters appear as larger length scales are considered.

The \emph{incipient infinite cluster} (IIC) was introduced in the physics literature \cite{AlexanderOrbach1982density,LeyvrazStanley1983class,StanleyConoglio1983fractal} as a way to analyse how this extensive structure ``emerges'' at the critical point. It can be viewed as a description of the not fully materialised infinite cluster. As explained in \cite{AizenmanNumberIncipient1997}, a rigorous definition of this object is not clear, and one may obtain different results depending on the chosen approach. 

Various methods have been suggested to construct the IIC. In a seminal paper, Kesten \cite{Kesten1986incipient} proposed two ways of constructing the IIC in the case of planar Bernoulli percolation. One consists in conditioning the critical measure $\mathbb P_{p_c}$ on the one-arm event $\{0\connect{}\Lambda_n^c\}$, where $\Lambda_n={\color{black}[-n,n]^d}\cap \mathbb Z^d$, and to send $n$ to infinity. The other one consists in conditioning $0$ to lie in an infinite cluster in the supercritical measure $\mathbb P_p$ with $p>p_c$, and to send the parameter $p$ to $p_c$. Both approaches were shown to converge (weakly) to the same percolation measure, which additionally satisfies almost surely the property that $0$ lies in an infinite cluster. The cluster of the origin under this measure is a natural candidate for the IIC. Chayes, Chayes, and Durett \cite{ChayesChayesDurrett1987inhomogeneouspercoIIC} proposed a different construction based on the introduction of a well-chosen \emph{non-homogeneous} bond parameter $p$. Subsequent works by Járai \cite{Jarai2003incipient,Jarai2003invasion} provided alternative but equivalent constructions of Kesten's infinite incipient cluster, thus comforting the robustness of the two dimensional IIC for independent bond percolation.

Building on the development of the \emph{lace expansion} \cite{BrydgesSpencerSAW,HaraSlade1990Perco,HaraSlade1990LatticeTrees,HaraSlade1992SAW}, van der Hofstad and Járai \cite{vdHofstadJarai2004IIC} proposed a construction of the IIC in the high-dimensional setup (i.e.~${d>10}$, see \cite{HaraSlade1990Perco,HaraDecayOfCorrelationsInVariousModels2008,FitznervdHofstad2017Perco-d>10}). In their approach, they constructed the limit of the sequence of measures $\mathbb P_{p_c}[\:\cdot\: |\: 0\connect{}x]$ as $|x|\rightarrow \infty$. They also proposed a \emph{subcritical} construction of the measure which complements Kesten's supercritical definition. These formulations are equivalent to Kesten's when $d=2$. An additional construction, involving conditioning on the one-arm event, was successfully achieved in \cite{HeydenreichvdHofstadHulshof2014IICrevisited} using the computation of the one-arm exponent performed by Kozma and Nachmias \cite{KozmaNachmias2011OneArm}. {\color{black}Recently, a construction of the IIC in dimensions $d>6$ was performed in \cite{chatterjee2025robust} using as main assumption an up-to-constant estimate on the critical two-point function.} Moreover, the lace expansion methods were used to construct the IIC in the setup of oriented percolation \cite{vdHofstadHollanderSlade2002IICorientedPerco}.
 
Fine properties of the IIC were derived using the aforementioned techniques: sub-diffusivity of the random-walk on the IIC has been studied in \cite{Kesten1986subdiffusiveonIIC,BarlowJaraoKumagaiSlade2008random,KozmaNachmias2009AlexanderOrbach}, the connection to invasion percolation has been explored in \cite{Jarai2003invasion,vdHofstadJarai2004IIC}, and the relation of the scaling limit of the high-dimensional IIC to the super-Brownian excursion has been investigated in \cite{HaraSlade2000scalingofIIC-I,HaraSlade2000scalingofIIC-II}.

Removing independency makes the problem more intricate. In the setup of the planar \emph{random cluster model} (or \emph{Fortuin--Kasteleyn (FK) percolation}) with cluster weight ${1< q \leq 4}$ ($q=1$ corresponds to Bernoulli percolation), the Russo--Seymour--Welsh (RSW) theory \cite{Russo1978note,SeymourWelsh1978percolation} is sufficient to extend Kesten's argument. See \cite{Kesten1986incipient,Jarai2003invasion,BasuSapo2017KestenIIC}. In higher dimensions, the lace expansion--- which is very useful to build the IIC for Bernoulli percolation--- has only been developed for $q\in\{1,2\}$. Setting $q=2$ corresponds to the so-called \emph{FK-Ising model}. However, the lace expansion of the Ising model is developed at the level of the spin model and very little information is available on the percolation side. In particular, it is not clear how to extend the approach of \cite{vdHofstadJarai2004IIC}. Nevertheless, other tools are available for the analysis of the case $q=2$.

{\color{black}
Finally, let us mention that the existence of the IIC was studied in other ``correlated'' settings including the percolation of the Gaussian free field on the cable graph, see \cite{cai2024incipient,werner2025switching}.
}
%


The goal of this paper is twofold. First, we study the existence of the IIC measure of the FK-Ising model in dimensions $d\geq 3$. Using the relation of the FK-Ising model to the \emph{random current representation} of the Ising model established in \cite{lupu2016note} (see also \cite{AizenmanDuminilTassionWarzelEmergentPlanarity2019}), together with the mixing property recently obtained in \cite{AizenmanDuminilTriviality2021}, we construct in two different ways a natural candidate for the IIC measure of the FK-Ising model in dimensions $d\geq 3$. Our definitions are similar to the ones suggested in \cite{vdHofstadJarai2004IIC}. Along the way, we also construct the corresponding object for the percolation measure induced by random currents. 

Second, we study the high-dimensional susceptibility of the Ising model and obtain its exact asymptotics as $\beta\nearrow \beta_c$, i.e. we prove that the limit  $\lim_{\beta\nearrow \beta_c}\chi(\beta)(1-\beta/\beta_c)$ exists. This improves a result of Aizenman \cite{AizenmanGeometricAnalysis1982}, where up-to-constants estimates were derived.
 We then relate this limit to the IIC measure of the \emph{double} random current.
 
 We stress that our proofs do not rely on the lace expansion. 
\subsection{Definition of the model} The Ising model is one of the most studied models in statistical mechanics (see \cite{duminilreviewIsing} for a review). It is formally defined as follows: given $\Lambda\subset \mathbb Z^d$ finite, $\tau\in\{-1,0,+1\}^{\mathbb Z^d}$, and an inverse temperature $\beta\geq 0$, construct a probability measure on $\{-1,+1\}^\Lambda$ according to the formula
\begin{equation}
    \langle F(\sigma)\rangle_{\Lambda,\beta}^\tau:=\frac{1}{\mathbf{Z}_{\Lambda,\beta}^\tau}\sum_{\sigma\in \{-1,+1\}^\Lambda}F(\sigma)\exp\Big(\beta\sum_{\substack{x,y\in \Lambda\\ x\sim y}}\sigma_x\sigma_y+\beta\sum_{\substack{x\in \Lambda\\y\notin \Lambda\\x\sim y}}\sigma_x\tau_y\Big),
\end{equation}
where ${F:\{-1,+1\}^\Lambda\rightarrow \mathbb R}$, $\mathbf{Z}_{\Lambda,\beta}^\tau$ is the \emph{partition function} of the model which guarantees that $\langle 1\rangle_{\Lambda,\beta}^\tau=1$, and ${x\sim y}$ means that ${|x-y|_2=1}$ (where $|\cdot|_2$ is the $\ell^2$ norm on $\mathbb R^d$). This definition corresponds to the \emph{nearest-neighbour ferromagnetic} Ising model on $\Lambda$ with \emph{boundary condition} $\tau$. When $\tau\equiv 1$ (resp $\tau\equiv 0$), we write $\langle \cdot \rangle^\tau_{\Lambda,\beta}=\langle \cdot \rangle^+_{\Lambda,\beta}$ (resp. $\langle \cdot\rangle_{\Lambda,\beta}$).

It is a well-known fact (see \cite{GriffithsCorrelationsIsing1-1967,GriffithsCorrelationsIsing2-1967}) that the measures $\langle \cdot\rangle_{\Lambda,\beta}^+$ and $\langle \cdot\rangle_{\Lambda,\beta}$ admit weak limits when $\Lambda\nearrow\mathbb Z^d$. We respectively denote them by $\langle \cdot \rangle_{\beta}^+$ and $\langle \cdot\rangle_{\beta}$. When $d\geq 2$, the model undergoes a \emph{phase transition} for the vanishing of the \emph{magnetisation} at a critical parameter $\beta_c\in(0,\infty)$ defined by
\begin{equation}
	\beta_c:=\inf\Big\{\beta\geq 0, \: m^*(\beta):=\langle \sigma_0\rangle_\beta^+>0\Big\}.
\end{equation}
The phase transition is \emph{continuous} in the sense that $m^*(\beta_c)=0$, see \cite{AizenmanFernandezCriticalBehaviorMagnetization1986,WernerPercolationEtModeledIsing2009,AizenmanDuminilSidoraviciusContinuityIsing2015,DuminilLecturesOnIsingandPottsModels2019}.

The Ising model is classically related to a dependent percolation model: the \emph{FK-Ising model}, which we now define. For more information we refer to the monograph \cite{Grimmett2006RCM} or the lecture notes \cite{DuminilLecturesOnIsingandPottsModels2019}.

Define $\mathbb E:=\{\{x,y\}: x\sim y, \: x,y\in \mathbb Z^d\}$. Let $G=(V,E)$ be a finite subgraph of $\mathbb Z^d$. Fix $\xi\in \{0,1\}^{\mathbb E\setminus E}$. For a fixed percolation configuration $\omega \in \{0,1\}^E$, we let $k^\xi(\omega)$ be the number of connected components of the graph with vertex set $\mathbb Z^d$ and edge set $\omega \vee \xi$ that intersect $G$. The FK-Ising model at inverse temperature $\beta\geq 0$ with boundary condition $\xi$ is the measure on $\{0,1\}^E$ defined by
\begin{equation}
	\phi^\xi_{G,\beta}[\omega]:=\frac{2^{k^\xi(\omega)}\prod_{\{x,y\}\in E}(e^{2\beta}-1)^{\omega(\{x,y\})}}{\mathbf{Z}_{G,\beta,\xi}^{{\rm FK}}},
\end{equation}
where $\mathbf{Z}_{G,\beta,\xi}^{{\rm FK}}$ is the partition function of the model, which ensures that $\phi^\xi_{G,\beta}$ is a probability measure. We let $\phi^1_{G,\beta}$ and $\phi_{G,\beta}^0$ respectively denote the cases $\xi\equiv 1$ and $\xi\equiv 0$. It is classical that the Ising model and the FK-Ising model can be coupled together. This leads to the following formulas: for $x,y\in G$, 
\begin{equation}\label{eq: consequences of ESC}
	\langle \sigma_x\sigma_y\rangle_{G,\beta}=\phi_{G,\beta}^0[x\connect{}y], \quad \langle \sigma_x\sigma_y\rangle_{G,\beta}^+=\phi_{G,\beta}^1[x\connect{}y], \quad \langle \sigma_x\rangle_{G,\beta}^+=\phi^1_{G,\beta}[0\connect{}\partial G],
\end{equation}
where $\partial G$ is the vertex boundary of $G$ defined by $\partial G:=\{x\in G: \exists y \notin G, \: x\sim y\}$.
Once again, it is possible to construct (weak) limits of these measures when $G\nearrow\mathbb Z^d$. We denote them by $\phi^1_\beta$ and $\phi^0_\beta$. In particular, $m^*(\beta)=\phi^1_{\beta}[0\connect{}\infty]$. The FK-Ising model undergoes a phase transition for the existence of an infinite cluster at the parameter $\beta_c$ introduced above. That is, an alternative definition of $\beta_c$ is
\begin{equation}
	\beta_c=\inf\Big\{\beta \geq 0 : \phi^1_\beta[0\connect{}\infty]>0\Big\}.
\end{equation}
It is known (see \cite{BodineauTranslationGibbsIsing2006,AizenmanDuminilSidoraviciusContinuityIsing2015,RaoufiGibbsMeasures}) that $\phi_{\beta}^1=\phi_{\beta}^0$ for all $\beta>0$. We denote the common measure $\phi_\beta$. The relation described above ensures that there is no infinite cluster at criticality: $\phi_{\beta_c}[0\connect{}\infty]=0$. As for Bernoulli percolation, defining a notion of critical measure conditioned on the event $\{0\connect{}\infty\}$ is therefore not clear. 

\subsection{Construction of the incipient infinite cluster}

The construction of the IIC for the FK-Ising model when $d=2$ is a classical consequence of RSW theory, see \cite{DuminilKKMO2020rotationalinv} and references therein. Our first result provides two equivalent constructions of the IIC measure in dimensions $d\geq 3$. We begin by introducing some useful terminology. Let $\mathcal F_0$ be the set of cylinder events on $\{0,1\}^\mathbb E$, i.e. the set of \emph{local} events that are measurable with respect to the state of finitely many edges. Let $\mathcal F$ be the $\sigma$-algebra generated by $\mathcal F_0$.
\begin{Thm}\label{thm: first construction IIC fk} Let $d\geq 3$. For all $E\in \mathcal F_0$, the limit
\begin{equation}
	\lim_{|x|\rightarrow \infty}\phi_{\beta_c}[E\: |\: 0\connect{}x]
\end{equation}
exists independently of the manner in which $x$ goes to infinity. We denote it by $\phi^\infty[E]$. Moreover, $\phi^\infty$ extends to a probability measure on $\mathcal F$, that we still denote $\phi^\infty$, and which satisfies
\begin{equation}
	\phi^\infty[0\connect{}\infty]=1.
\end{equation}
\end{Thm}
Note that the above measure is not translation invariant. The cluster of the origin under the measure $\phi^\infty$ is a natural candidate for the IIC of the FK-Ising model in dimensions $d\geq 3$. We now provide a \emph{subcritical} definition of $\phi^\infty$ that is inspired by \cite{vdHofstadJarai2004IIC}. 

For $\beta<\beta_c$, we define a probability measure $\mathbb Q_\beta$ on $\{0,1\}^{\mathbb E}$ by setting for $E\in \mathcal F$, 
\begin{equation}
	\mathbb Q_\beta[E]:=\frac{1}{\chi(\beta)}\sum_{x\in \mathbb Z^d}\phi_\beta[E, \:0\connect{}x],
\end{equation}
where $\chi(\beta):=\sum_{x\in \mathbb Z^d}\langle \sigma_0\sigma_x\rangle_\beta=\sum_{x\in \mathbb Z^d}\phi_\beta[0\connect{}x]$ is the \emph{susceptibility}.

\begin{Thm}\label{thm: subcritical approach to IIC FK Ising} Let $d\geq 3$. For all $E\in \mathcal F_0$, the limit
\begin{equation}
	\lim_{\beta\nearrow\beta_c}\mathbb Q_\beta[E]
\end{equation}
exists. We denote it by $\mathbb Q_{\beta_c}[E]$. Moreover, one has $\mathbb Q_{\beta_c}[E]=\phi^\infty[E]$.
\end{Thm}
It is not clear how to construct the IIC with Kesten's approach. We discuss this issue in Section \ref{section: open questions}.

Our strategy to prove these results is very similar to the one employed in \cite{vdHofstadJarai2004IIC}: we use a \emph{mixing property} to argue that for a local event $E$, the quantity $\phi_{\beta_c}[E\: |\: 0\connect{}x]$ ``does not depend much'' on $x$ when $x$ is sufficiently far away from the support of $E$. We are not able to directly prove such a mixing property at the level of the FK-Ising model, but we circumvent this difficulty by using an intermediate model related to the Ising model: the \emph{random current model}. We introduce this classical geometric representation of the Ising model in Section \ref{section: rcr}. In the recent breakthrough work of Aizenman and Duminil-Copin \cite{AizenmanDuminilTriviality2021}, a mixing property of the random current measure is derived and plays a pivotal role. In the same paper, the authors suggest that their result may be used to construct the IIC for the FK-Ising model in all dimensions $d\geq 3$. Although their mixing property is only proved in dimensions $d\geq 4$, they suggest a possible route (which relies on \cite{AizenmanDuminilSidoraviciusContinuityIsing2015}) to extend it to the three-dimensional case. We provide a full argument in Section \ref{section: mixing d=3}, see also Theorem \ref{thm: mixing for currents} for a statement. With this result, one may construct the IIC measure of the random current measure, see Theorem \ref{thm: IIC rcr}. 

The construction of the IIC for the FK-Ising measure then follows from an observation of Lupu and Werner \cite{lupu2016note} (see also \cite[Theorem~3.2]{AizenmanDuminilTassionWarzelEmergentPlanarity2019}) that identifies the law $\phi_{\beta_c}[\:\cdot\:|\:0\connect{}x]$ with that of a percolation configuration induced by a single current configuration on $\mathbb Z^d$ with sources $\{0,x\}$ \emph{sprinkled} by an independent (well-chosen) Bernoulli percolation. We refer to Proposition \ref{prop: coupling fk rcr} for a precise statement.

\subsection{Susceptibility of the high-dimensional Ising model}
We recall that the susceptibility is defined for $\beta<\beta_c$ by
\begin{equation}
	\chi(\beta)=\sum_{x\in \mathbb Z^d}{\color{black}\langle \sigma_0\sigma_x\rangle_{\beta}}.
\end{equation}
It is well-known that the nearest-neighbor Ising model introduced above is \emph{reflection positive}, see \cite{FrohlichSimonSpencerIRBounds1976,FrohlichIsraelLiebSimon1978,BiskupReflectionPositivity2009}. This property provides two tools that are fundamental in the study of the Ising model: the \emph{infrared bound} \cite{FrohlichSimonSpencerIRBounds1976,FrohlichIsraelLiebSimon1978} and the Messager--Miracle-Solé (MMS) inequalities \cite{MessagerMiracleSoleInequalityIsing}. These tools combined together imply the existence of $C=C(d)>0$ such that for all $\beta\leq \beta_c$, and all $x\in \mathbb Z^d\setminus\{0\}$,
\begin{equation}\label{eq: IRB}
	\langle \sigma_0\sigma_x\rangle_\beta\leq \frac{C}{|x|^{d-2}},
\end{equation}
see \cite{PanisTriviality2023} for more details. When $d>4$, a matching lower bound (with a different constant) was recently derived \cite{duminil2024new}.

In his seminal paper, Aizenman \cite{AizenmanGeometricAnalysis1982} used \eqref{eq: IRB} to obtain the existence of $C>0$ such that for all $\beta<\beta_c$,
\begin{equation}\label{eq: bounds on the susceptibility}
	\frac{(2d\beta_c)^{-1}}{1-\beta/\beta_c}\leq \chi(\beta)\leq \frac{C}{1-\beta/\beta_c}.
\end{equation}
In fact, for every $\varepsilon>0$, provided that $\beta$ is sufficiently close to $\beta_c$, $C$ can be chosen\footnote{It is possible to extract from their methods a slightly better bound, where the bubble is replaced by an \emph{open} bubble diagram. See the discussion below Theorem \ref{thm: value of the constant} and Remark \ref{rem: lower bound proba}.} \cite{AizenmanGrahamRenormalizedCouplingSusceptibilityd=4-1983} as
\begin{equation}
	C=\frac{1+(2d\beta_c)B(\beta_c)}{2d\beta_c}(1+\varepsilon),
\end{equation}
where
\begin{equation}
	B(\beta):=\sum_{x\in \mathbb Z^d}\langle \sigma_0\sigma_x\rangle_\beta^2
\end{equation}
is the so-called \emph{bubble diagram}. Note that it is finite when $d>4$ by \eqref{eq: IRB}. Our next result is a strengthening of \eqref{eq: bounds on the susceptibility}.

\begin{Thm}\label{thm: exact behaviour} Let $d>4$. There exists $A=A(d)>0$ such that, for $\beta<\beta_c$, 
\begin{equation}
	\chi(\beta)=\frac{A}{1-\beta/\beta_c}(1+o(1)),
\end{equation}
where $o(1)$ tends to $0$ as $\beta$ tends to $\beta_c$.
\end{Thm}
{\color{black}This result is new when considering all dimensions $d>4$. In large enough dimensions (meaning $d\gg 4$), Theorem \ref{thm: exact behaviour} can be recovered from the lace expansion analysis of \cite{Sakai2007LaceExpIsing,Sakai2022correctboundsIsing}. However, the geometric interpretation of the constant $A$ presented below is novel, even in the setting where lace expansion applies\footnote{The existence of $A$ has also been established through the lace expansion for the self-avoiding walk model and Bernoulli percolation in sufficiently high dimensions, see \cite{HaraSlade1990Perco,slade2006lace,clisby2007self}. It would be interesting to see whether it is related the IIC of the corresponding models as in the case of the Ising model.}.}

We write the asymptotic in terms of $\beta/\beta_c$ rather than $\beta_c-\beta$ as it makes appear a constant $A(d)$ which should tend to $1$ as $d$ tends to $\infty$. This can be justified by looking at the Curie--Weiss model (which corresponds to $d=\infty$), for which it is known that $A=1$ (see for instance \cite[Chapter~2]{FriedliVelenikIntroStatMech2017}). Below, we provide a proof of this convergence result.

 In fact, it is possible to relate the constant $A$ obtained in Theorem \ref{thm: exact behaviour} to the IIC of a \emph{double} random current. Let $\mathbf{e}_1$ be the unit vector with first coordinate equal to $1$. Let $\mathbf{P}^{0\infty,\mathbf{e}_1\infty}$ be the IIC measure defined in Theorem \ref{thm: IIC rcr}. Under this measure, $0$ lies in an infinite cluster made of an infinite path emerging from the origin and a collection of finite loops attached to this path. Using the terminology introduced in Section \ref{section: rcr}, the \emph{sources} of $\n_1$ can be viewed as ``$\sn_1=\{0,\infty\}$''. A similar observation can be made for $\mathbf{e}_1$. Hence, under this measure, two infinite structures emerge from two neighbouring points of $\mathbb Z^d$. The next result relates $A$ to the probability that these clusters avoid each other. If $(\n_1,\n_2)\sim \mathbf P^{0\infty,\mathbf{e}_1\infty}$, we let $\mathbf C_{\n_1+\n_2}(0)$ (resp. $\mathbf C_{\n_1+\n_2}(\mathbf{e}_1)$) be the (infinite) cluster of $0$ (resp. $\mathbf{e}_1$) in the percolation configuration induced by $\n_1+\n_2$.
\begin{Thm}\label{thm: value of the constant} Let $d>4$. Let $A$ be the constant of Theorem \textup{\ref{thm: exact behaviour}}. Then,
\begin{equation}
	A^{-1}=(2d\beta_c)\cdot\mathbf P^{0\infty,\mathbf{e}_1\infty}[\mathbf C_{\n_1+\n_2}(0)\cap \mathbf C_{\n_1+\n_2}(\mathbf{e}_1)=\emptyset].
\end{equation}
\end{Thm}

We obtain Theorems \ref{thm: exact behaviour} and \ref{thm: value of the constant} using the probabilistic interpretation of the derivative of the susceptibility obtained in \cite{AizenmanGeometricAnalysis1982}, see \eqref{eq: proof thm 1 (3)}. This result already makes appear the probability of occurrence of the event $\{\mathbf C_{\n_1+\n_2}(0)\cap \mathbf C_{\n_1+\n_2}(\mathbf{e}_1)=\emptyset\}$ under the following measure:
\begin{equation}
	\mathbf Q_\beta[\cdot]:=\frac{1}{\chi(\beta)^2}\sum_{x,y\in \mathbb Z^d}\langle \sigma_0\sigma_x\rangle_\beta\langle \sigma_{\mathbf{e}_1}\sigma_y\rangle_\beta\mathbf P_\beta^{0x,\mathbf{e}_1y}[\cdot],
\end{equation}
where $\mathbf P^{0x,\mathbf{e}_1y}_\beta$ samples independent currents $\n_1$ and $\n_2$ with respective source sets $\sn_1=\{0,x\}$ and $\sn_2=\{\mathbf{e}_1,y\}$. In Section \ref{section: iic}, we prove that the measure $\mathbf Q_\beta$ converges weakly as $\beta\nearrow \beta_c$ to the natural candidate for the IIC measure of the double random current. We also prove that convergence holds for \emph{local} events. However, the event of interest is not local. Our main contribution is a \emph{localization} of the event $\{\mathbf C_{\n_1+\n_2}(0)\cap \mathbf C_{\n_1+\n_2}(\mathbf{e}_1)=\emptyset\}$ which relies on the assumption $d>4$. See Lemma \ref{lem: localization step} for a precise statement.

\paragraph{Limit of $A(d)$ as $d$ tends to infinity.} Together with the results of \cite{AizenmanGrahamRenormalizedCouplingSusceptibilityd=4-1983}, we can prove (see Remark \ref{rem: lower bound proba}) that
\begin{equation}\label{eq: lower bound on P intro}
	\mathbf P^{0\infty,\mathbf{e}_1\infty}[\mathbf C_{\n_1+\n_2}(0)\cap \mathbf C_{\n_1+\n_2}(\mathbf{e}_1){\color{black}=\emptyset}]\geq \frac{1}{1+(2d\beta_c)B^{(\mathbf{e}_1)}(\beta_c)},
\end{equation}
where
\begin{equation}\label{eq: def open bubble}
	B^{(\mathbf{e}_1)}(\beta):=\sum_{x\in \mathbb Z^d}\langle \sigma_0\sigma_x\rangle_\beta \langle \sigma_{\mathbf{e}_1}\sigma_x\rangle_\beta
\end{equation}
is an \emph{open} bubble diagram. The lace expansion analysis of the Ising model (see \cite{Sakai2007LaceExpIsing,Sakai2022correctboundsIsing}) implies that
\begin{equation}\label{eq: lim beta_c d inf}
	\lim_{d\rightarrow \infty}2d\beta_c=1, \qquad \lim_{d\rightarrow \infty}B^{(\mathbf{e}_1)}(\beta_c)=0.
\end{equation}
Together with \eqref{eq: lower bound on P intro}, this gives
\begin{equation}\label{eq: lim P d inf}
	\lim_{d\rightarrow \infty}\mathbf P^{0\infty,\mathbf{e}_1\infty}[\mathbf C_{\n_1+\n_2}(0)\cap \mathbf C_{\n_1+\n_2}(\mathbf{e}_1)=\emptyset]=1.
\end{equation}
Plugging \eqref{eq: lim beta_c d inf} and \eqref{eq: lim P d inf} in Theorem \ref{thm: value of the constant} yields
\begin{equation}
	\lim_{d\rightarrow \infty}A(d)=1.
\end{equation}

{\color{black}
In Theorem \ref{thm: IIC rcr}, we will show the following weak convergence result:
\begin{equation}\label{eq:weakcvintro}
 \lim_{|x|,|y|\rightarrow \infty}\mathbf P_{\beta_c}^{0x,\mathbf{e}_1 y}=\mathbf P^{0\infty,\mathbf{e}_1\infty},
 \end{equation}
where $\mathbf P^{0x,\mathbf{e}_1y}_{\beta_c}$ is the double random current measure defined in Section \ref{section: rcr}. Since $\{\mathbf C_{\n_1+\n_2}(0)\cap \mathbf C_{\n_1+\n_2}(\mathbf{e}_1)=\emptyset\}$ is not a local event, the combination of \eqref{eq: lower bound on P intro} and \eqref{eq:weakcvintro} is not enough to get any lower bound on $\mathbf P^{0x,\mathbf{e}_1y}_{\beta_c}[\mathbf C_{\n_1+\n_2}(0)\cap \mathbf C_{\n_1+\n_2}(\mathbf{e}_1)=\emptyset]$ that is uniform in $x\in \mathbb Z^d\setminus\{\mathbf{e}_1\}$ and $y\in \mathbb Z^d\setminus\{0,x\}$. The following result provides a rather weak lower bound on the latter quantity (at least when $x$ and $y$ are not too small and not too close to each other).
\begin{Thm}\label{prop:uniform lower bound} Let $d>4$. There exist $c_0,N_0>0$ such that, for every $x,y\in \mathbb Z^d$ satisfying $|x|\wedge |y|\wedge |x-y|\geq N_0$, one has
\begin{equation}
\mathbf P^{0x,\mathbf{e}_1y}_{\beta_c}[\mathbf C_{\n_1+\n_2}(0)\cap \mathbf C_{\n_1+\n_2}(\mathbf{e}_1)=\emptyset]\geq c_0.
\end{equation}
\end{Thm}
} 
\subsection{Open problems}\label{section: open questions}

To conclude this introduction, we list a few questions which naturally arise from the above results.

\subsubsection{Other definitions of the IIC}

In view of the above, it is natural to ask whether Kesten's approach can be used to construct the IIC of the FK-Ising model. This leads to the following questions.

\begin{Question} Let $d\geq 3$. Show that one may define the measure $\phi^\infty$ by conditioning on the one-arm event:
\begin{equation}
	\phi^\infty=\lim_{n\rightarrow \infty}\phi_{\beta_c}[\:\cdot \: |\: 0\connect{} \partial \Lambda_n]
\end{equation}
\end{Question}
In \cite{HeydenreichvdHofstadHulshof2014IICrevisited} this construction was successfully achieved in the context of high-dimensional Bernoulli percolation using the computation of the one-arm exponent\footnote{The one-arm exponent $\rho$ of Bernoulli percolation is defined by $\mathbb P_{p_c}[0\connect{} \partial \Lambda_n]\asymp n^{-1/\rho}$. It is known to be equal to $1/2$ when $d>10$.} performed in \cite{KozmaNachmias2011OneArm} (to be more precise, they only need a lower bound on this exponent). To the best of our knowledge, a lower bound on this exponent, even in the mean-field regime, is not known (see \cite{SakaiHandaHeydenreich2019oneArmIsing} for an upper bound).
\begin{Question} Let $d\geq 3$. Show that one may define the measure $\phi^\infty$ from the supercritical regime:
\begin{equation}
	\phi^\infty=\lim_{\beta\searrow \beta_c}\phi_{\beta}[\:\cdot \: |\: 0\connect{} \infty]
\end{equation}
\end{Question}
The proof of Theorem \ref{thm: subcritical approach to IIC FK Ising} relies on the near-critical version of the mixing property of \cite{AizenmanDuminilTriviality2021}. Such a near-critical tool is not available on the supercritical side and this seems to be the main obstacle to follow this approach.

\subsubsection{Behaviour of the susceptibility when $d=4$}

In view of Theorem \ref{thm: exact behaviour}, it is natural to ask what is the behaviour of the susceptibility when $d=4$, where mean-field behaviour still holds. Renormalization group analysis and universality hypothesis (see for instance \cite{BauerschmidtBrydgesSlade2014Phi4fourdim} for the corresponding result in the setup of the weakly-coupled $\varphi^4$ model) suggest the following behaviour.
\begin{Conj} Let $d=4$. There exists $A>0$ such that
\begin{equation}
	\chi(\beta)=\frac{A|\log (1-\beta/\beta_c)|^{1/3}}{1-\beta/\beta_c}(1+o(1)),
\end{equation}
where $o(1)$ tends to $0$ as $\beta$ tends to $\beta_c$.
\end{Conj}

The best result regarding this conjecture was obtained in \cite{AizenmanGrahamRenormalizedCouplingSusceptibilityd=4-1983} where the authors obtained the existence of $C>0$ such that for all $\beta<\beta_c$,
\begin{equation}
	\frac{(2d\beta_c)^{-1}}{1-\beta/\beta_c}\leq \chi(\beta)\leq \frac{C|\log (1-\beta/\beta_c)|}{1-\beta/\beta_c}.
\end{equation}

In the four-dimensional case, Theorem \ref{thm: value of the constant} suggests that the divergence of $\chi(\beta)(1-\beta/\beta_c)$ as $\beta$ tends to $\beta_c$ is related to the fact that $\mathbf P^{0\infty,\mathbf{e}_1\infty}[\mathbf C(0)\cap \mathbf C(\mathbf{e}_1)=\emptyset]=0$. This hints that $d=4$ is (just like for random walks) the \emph{critical dimension} for the intersection of two neighbouring random current IICs. As we will see in the proof of Theorem \ref{thm: value of the constant}, when $d>4$ it is possible to relate the susceptibility to $\mathbf P^{0\infty,\mathbf{e}_1\infty}[\mathbf C(0)\cap \mathbf C(\mathbf{e}_1)\cap \Lambda_k=\emptyset]+o(1)$ where $o(1)$ is uniform in $\beta$ and tends to $0$ as $k$ tends to infinity. When $d=4$, this ``$o(1)$'' is not uniform in $\beta$ anymore. One potential idea to tackle this problem is to choose $k=k(\beta)$ as a function of $\beta$. It seems natural to choose $k(\beta)\geq\xi(\beta)$ where we recall that $\xi(\beta)$ is the so-called \emph{correlation length}, defined by
\begin{equation}\label{eq:correlation length}
	\xi(\beta)^{-1}:=\lim_{n\rightarrow \infty} -\frac{1}{n}\log \langle \sigma_0\sigma_{n\mathbf{e}_1}\rangle_\beta.
\end{equation}

\begin{Question} When $d=4$, relate the asymptotic behaviour of the susceptibility to the probability of avoidance of two neighbouring IICs up to distance $\xi(\beta)^C$, where  $C\geq 1$.
\end{Question}

\section{The random current representation}\label{section: rcr}
We begin by introducing the random current representation of the Ising model. For more information we refer to \cite{DuminilLecturesOnIsingandPottsModels2019} or \cite[Chapter~2]{panis2024applications}. \subsection{Definition and main properties}
Let $\Lambda$ be a finite subset of $\mathbb Z^d$.
\begin{Def} A \textit{current} $\n$ on $\Lambda$ is a function defined on the edge-set $E(\Lambda):=\lbrace \lbrace x,y\rbrace: \: x,y \in \Lambda \textup{ and }x\sim y \rbrace$ and taking its values in $\mathbb N=\lbrace 0,1,\ldots\rbrace$. We let $\Omega_\Lambda$ be the set of currents on $\Lambda$. The set of \emph{sources} of $\n$, denoted by $\sn$, is defined as
\begin{equation}
\sn:=\Big\{ x \in \Lambda \: : \: \sum_{y\sim x}\n_{x,y}\textup{ is odd}\Big\}.
\end{equation}
We also define 
\begin{equation}
w_\beta(\n):=\prod_{\substack{\lbrace x,y\rbrace\in E(\Lambda)}}\dfrac{\beta^{\n_{x,y}}}{\n_{x,y}!}.
\end{equation}
\end{Def}
It is classical \cite{AizenmanGeometricAnalysis1982,DuminilLecturesOnIsingandPottsModels2019} that the correlation functions of the Ising model are related to currents: for
$S\subset \Lambda$, if $\sigma_S:=\prod_{x\in S}\sigma_x$,

\begin{equation}\label{equation correlation rcr}
\left\langle \sigma_S\right\rangle_{\Lambda,\beta}=\dfrac{\sum_{\sn=S}w_\beta(\n)}{\sum_{\sn=\emptyset}w_\beta(\n)}.
\end{equation}

The trace of a current $\n$ naturally induces a percolation configuration $(\mathds{1}_{\n_{x,y}>0})_{\{x,y\}\in E(\Lambda)}$ on $\Lambda$. As it turns out, the connectivity properties of this percolation model play a crucial role in the analysis of the Ising model. This motivates the following terminology.

\begin{Def} Let $\n\in \Omega_{\Lambda}$ and $x,y\in \Lambda$.
\begin{enumerate}
    \item[$(i)$] We say that $x$ is \emph{connected} to $y$ in $\n$ and write $x\overset{\n}{\longleftrightarrow} y$, if there is a sequence of points $x_0=x,x_1,\ldots, x_m=y$ such that $\n_{x_i,x_{i+1}}>0$ for $0\leq i \leq m-1$.
    \item[$(ii)$] The \emph{cluster} of $x$, denoted by $\mathbf{C}_\n(x)$, is the set of points connected to $x$ in $\n$.
\end{enumerate}
\end{Def}

The main interest of the random current representation lies in the following combinatorial tool, called the \emph{switching lemma}, which was first introduced in \cite{GriffithsHurstShermanConcavity1970}. The probabilistic picture attached to it was developed in \cite{AizenmanGeometricAnalysis1982}.
\begin{Lem}[Switching Lemma]\label{lem: switching lemma}
For every $S_1,S_2\subset \Lambda$ and every function $F$ from $\Omega_\Lambda$ into $\mathbb R$,
\begin{multline}\label{eq: switching lemma}
\sum_{\substack{\n_1\in \Omega_\Lambda : \:\sn_1=S_1\\ \n_2\in \Omega_\Lambda :\: \sn_2=S_2}}F(\n_1+\n_2)w_\beta(\n_1)w_\beta(\n_2)\\=\sum_{\substack{\n_1\in \Omega_\Lambda: \:\sn_1=S_1\Delta S_2\\ \n_2\in \Omega_\Lambda :\: \sn_2=\emptyset}}F(\n_1+\n_2)w_\beta(\n_1)w_\beta(\n_2)\mathds{1}_{(\n_1+\n_2)\in \mathcal{F}_{S_2}},
\end{multline}
where $S_1\Delta S_2=(S_1\cup S_2)\setminus (S_1\cap S_2)$ is the symmetric difference of sets and $\mathcal{F}_{S_2}$ is given by
\begin{equation}\label{eq: event F_B}
\mathcal{F}_{S_2}=\lbrace \n \in \Omega_\Lambda : \: \exists \mathbf{m}\leq \n \:, \: \partial \mathbf{m}=S_2\rbrace.
\end{equation}
\end{Lem}
{\color{black}
\begin{Rem}\label{rem: fs is a perco event} Observe that the event $\mathcal F_{S_2}$ introduced in \eqref{eq: event F_B} is in fact a percolation event. Indeed, it is ex{black}actly the event that each connected component of the percolation configuration induced by $\n$ intersects $S_2$ an even number of times.
\end{Rem}
}
If $S\subset \Lambda$, define a probability measure $\mathbf{P}_{\Lambda,\beta}^S$ on $\Omega_\Lambda$ as follows: for every $\n\in \Omega_\Lambda$, 
\begin{equation}
    \mathbf{P}_{\Lambda,\beta}^S[\n]:=\mathds{1}_{\sn=S}\frac{w_\beta(\n)}{Z_{\Lambda,\beta}^S},
\end{equation}
where $Z_{\Lambda,\beta}^S:=\sum_{\sn=S}w_\beta(\n)$ is a normalisation constant. Moreover, for $S_1,\ldots,S_k\subset \Lambda$, define
\begin{equation}
    \mathbf{P}_{\Lambda,\beta}^{S_1,\ldots,S_k}:=\mathbf{P}_{\Lambda,\beta}^{S_1}\otimes\ldots\otimes \mathbf{P}_{\Lambda,\beta}^{S_k}.
\end{equation}
When $S=\lbrace x,y\rbrace$, we write $\mathbf{P}^{xy}_{\Lambda,\beta}$ instead of $\mathbf{P}^{\{x,y\}}_{\Lambda,\beta}$. If $\mathcal{E}\subset \Omega_\Lambda$, we also write $Z_{\Lambda,\beta}^S[\mathcal{E}]:=\sum_{\sn=S}w_\beta(\n)\mathds{1}_{\n\in \mathcal{E}}$. As proved in \cite{AizenmanDuminilSidoraviciusContinuityIsing2015}, if $S$ is a finite even (i.e. $|S|$ is even) subset of $\mathbb Z^d$, the sequence of probability measures $(\mathbf{P}_{\Lambda,\beta}^S)_{\Lambda\subset \mathbb Z^d}$ admits a weak limit as $\Lambda\nearrow \mathbb Z^d$ that we denote by $\mathbf{P}_{\beta}^S$.

The following mixing statement will be the main tool to construct the incipient infinite cluster.
\begin{Thm}[Mixing property of currents]\label{thm: mixing for currents} Let $d\geq 3$ and $s\geq 1$. Let $\varepsilon>0$. For every $n\geq 1$, there exist $N\geq n$ and $\beta(\varepsilon)<\beta_c$ such that, for every $1\leq t \leq s$, every $\beta\in [\beta(\varepsilon),\beta_c]$, every $x_i\in \Lambda_n$ and $y_i\notin \Lambda_N$ $(i\leq t)$, and every events $E$ and $F$ depending on the restriction of $(\n_1,\ldots,\n_s)$ to edges with endpoints within $\Lambda_n$ and outside $\Lambda_N$ respectively,
\begin{equation}\label{eq 1 mixing}
    \left|\mathbf{P}_\beta^{x_1 y_1,\ldots, x_t y_t,\emptyset,\ldots,\emptyset}[E\cap F]-\mathbf{P}_\beta^{x_1 y_1,\ldots, x_t y_t,\emptyset,\ldots,\emptyset}[E]\mathbf{P}_\beta^{x_1 y_1,\ldots, x_t y_t,\emptyset,\ldots,\emptyset}[F]\right|\leq \varepsilon.
\end{equation}
Furthermore, for every $x_1',\ldots, x_t'\in \Lambda_n$ and $y_1',\ldots,y'_t\notin \Lambda_N$, we have that
\begin{equation}\label{eq 2 mixing}
    \left|\mathbf{P}_\beta^{x_1 y_1,\ldots, x_t y_t,\emptyset,\ldots,\emptyset}[E]-\mathbf{P}_\beta^{x_1 y'_1,\ldots, x_t y'_t,\emptyset,\ldots,\emptyset}[E]\right|\leq \varepsilon,
\end{equation}
\begin{equation}\label{eq 3 mixing}
    \left|\mathbf{P}_\beta^{x_1 y_1,\ldots, x_t y_t,\emptyset,\ldots,\emptyset}[F]-\mathbf{P}_\beta^{x'_1 y_1,\ldots, x'_t y_t,\emptyset,\ldots,\emptyset}[F]\right|\leq \varepsilon.
\end{equation}
\end{Thm}
We will only use \eqref{eq 2 mixing}. However, for sake of completeness, we state and prove the stronger result.

Theorem \ref{thm: mixing for currents} was obtained by Aizenman and Duminil-Copin \cite{AizenmanDuminilTriviality2021} to derive the \emph{marginal triviality} of the scaling limits of Ising and $\varphi^4$ systems in dimension four. Their result is even stronger in the sense that it is \emph{quantitative}. However, their argument only works in dimensions $d\geq 4$. This is mainly due to the fact that the infrared bound \eqref{eq: IRB} is not sharp when $d=3$. As they mention (see discussion below \cite[Theorem~6.4]{AizenmanDuminilTriviality2021}), it is possible to extend their strategy to the three-dimensional case to the cost of losing the quantitative statement. We provide a full proof in Section \ref{section: mixing d=3}, which heavily relies on the fact that $m^*(\beta_c)=0$.

\subsection{Coupling with the FK-Ising model}
The following result can be found in \cite[Theorem~3.2]{AizenmanDuminilTassionWarzelEmergentPlanarity2019} {\color{black}(see also \cite{hansen2025uniform,hansen2025general})} and originates from the observation of \cite{lupu2016note}. {\color{black}Recall that the event $\mathcal F_S$ introduced in \eqref{eq: event F_B} can be viewed as a percolation event thanks to Remark \ref{rem: fs is a perco event}. This allows to make sense of $\phi^0_{G,\beta}[\:\cdot\:|\: \mathcal F_S]$ below.}

\begin{Prop}\label{prop: coupling fk rcr} Let $G=(V,E)$ be a subgraph of $\mathbb Z^d$ and let $S\subset G$ be a finite even subset of $G$. Let $\beta>0$. Let $\n$ be distributed according to $\mathbf P_{G,\beta}^{S}$. Let $(\omega_e)_{e\in \mathbb E}$ be an independent Bernoulli percolation with parameter $1-\exp(-\beta)$. For each $e\in E$, we define
\begin{equation}
	\eta_e:=\max(\mathds{1}_{\n_e>0},\omega_e).
\end{equation}
Then, the law of $\eta$ is exactly $\phi_{G,\beta}^0[\:\cdot\:|\: {\color{black}\mathcal F_S}]$. In particular, if $\mathcal A$ is an increasing event,
\begin{equation}
	\mathbf P^{S}_{G,\beta}[\mathcal A]\leq \phi_{G,\beta}^0[\mathcal A\: | \: \mathcal F_S].
\end{equation}
\end{Prop}
We will apply the above statement to the special cases $S=\emptyset$ or $S=\{x,y\}$ for $x,y\in G$. In the former case the law of $\eta$ is $\phi^0_{G,\beta}$, while in the latter case it is $\phi^0_{G,\beta}[\:\cdot\:|\:x\connect{}y]$, i.e. the FK-Ising measure conditioned on the event that $x$ is connected to $y$. Finally, when $G=\mathbb Z^d$, by the results recalled above, $\phi^0_{G,\beta}=\phi_\beta$.
\section{The incipient infinite cluster}\label{section: iic}
In this section, we construct the IIC measure for both random currents and the FK-Ising model.
\subsection{Construction of the IIC measure for random currents}
We will prove the following result.
\begin{Thm}[The IIC measure of the random current]\label{thm: IIC rcr}Let $d\geq 3$ and $s\geq 1$. Let $1\leq t \leq s$. For all $x_1,\ldots,x_t\in \mathbb Z^d$ there exists a measure $\mathbf P^{x_1\infty,\ldots,x_t\infty,\emptyset,\ldots,\emptyset}$ (where $\emptyset$ appears $s-t$ times) on $\Omega_{\mathbb Z^d}$ such that, for all local events $\mathcal A\subset (\Omega_{\mathbb Z^d})^s$,
\begin{equation}\label{eq: IIC critical}
	\lim_{|y_1|,\ldots, |y_t|\rightarrow \infty}\mathbf P_{\beta_c}^{x_1y_1,\ldots,x_ty_t,\emptyset,\ldots,\emptyset}[\mathcal A]=\mathbf P^{x_1\infty,\ldots,x_t\infty,\emptyset,\ldots,\emptyset}[\mathcal A],
\end{equation}
regardless of the manner in which $y_1,\ldots,y_t$ are sent to infinity.
Moreover, one also has,
\begin{equation}\label{eq: IIC near critical}
	\lim_{\beta\nearrow \beta_c}\frac{1}{\chi(\beta)^t}\sum_{y_1,\ldots,y_t\in \mathbb Z^d}\langle \sigma_{x_1}\sigma_{y_1}\rangle_\beta\ldots \langle \sigma_{x_t}\sigma_{y_t}\rangle_\beta\mathbf P_{\beta}^{x_1y_1,\ldots,x_ty_t,\emptyset,\ldots,\emptyset}[\mathcal A]=\mathbf P^{x_1\infty,\ldots,x_t\infty,\emptyset,\ldots,\emptyset}[\mathcal A].
\end{equation}
Finally, the measure $\mathbf P^{x_1\infty,\ldots,x_t\infty,\emptyset,\ldots,\emptyset}$ satisfies,
\begin{equation}
	\mathbf P^{x_1\infty,\ldots,x_t\infty,\emptyset,\ldots,\emptyset}[x_i\connect{\n_i\:}\infty, \: \forall 1\leq i \leq t]=1.
\end{equation}
\end{Thm}
{\color{black}
\begin{Rem} In a sampling $(\n_1,\ldots,\n_s)$ of the measure $\mathbf P^{x_1\infty,\ldots,x_t\infty,\emptyset,\ldots,\emptyset}$ introduced above, each $x_i$ ($1\leq i \leq t$) is a source of $\n_i$. It is paired with a ``source at infinity''.
\end{Rem}
}
\begin{proof}[Proof of Theorem \textup{\ref{thm: IIC rcr}}] We let $\mathcal A$ be a local event. Let $n\geq 1$ such that $\mathcal A$ is measurable in terms of the edges with endpoints within $\Lambda_n$.

Let $\varepsilon>0$. Using Theorem \ref{thm: mixing for currents} , we find $N$ large enough such that for all $y_1,y_1',\ldots,y_t,y_t'\notin \Lambda_N$,
\begin{equation}
	\Big|\mathbf P_{\beta_c}^{x_1y_1,\ldots,x_ty_t,\emptyset,\ldots,\emptyset}[\mathcal A]-\mathbf P_{\beta_c}^{x_1y_1',\ldots,x_ty_t',\emptyset,\ldots,\emptyset}[\mathcal A]\Big|\leq \varepsilon.
\end{equation}
This shows that the sequence $(\mathbf P^{x_1y_1,\ldots,x_ty_t,\emptyset,\ldots,\emptyset}_{\beta_c}[\mathcal A])_{y_1,\ldots,y_t}$ is Cauchy\footnote{By this, we mean that for any enumerations $r^{(i)}$, ${1\leq i \leq t}$, of $\mathbb Z^d$, the sequence $(\mathbf P^{x_1r^{(1)}_n,\ldots,x_tr^{(t)}_n,\emptyset,\ldots,\emptyset}_{\beta_c}[\mathcal A])_{n\geq 1}$ is Cauchy, and all the limits are the same.} and thus admits a limit that we denote by $\mathbf P^{x_1\infty,\ldots,x_t\infty,\emptyset,\ldots,\emptyset}[\mathcal A]$. This allows to define the measure $\mathbf P^{x_1\infty,\ldots,x_t\infty,\emptyset,\ldots,\emptyset}$ on the cylinder $\sigma$-algebra of $\Omega_{\mathbb Z^d}$.

We turn to the proof of \eqref{eq: IIC near critical}. Let $\varepsilon>0$. By Theorem \ref{thm: mixing for currents} we have the existence of $N$ large enough and $\beta(\varepsilon)$ close enough to $\beta_c$, such that for all $\beta\in[\beta(\varepsilon),\beta_c]$, for all $y_1,y_1',\ldots,y_t,y_t'\notin \Lambda_N$,
\begin{equation}
	\Big|\mathbf {P}_{\beta}^{x_1y_1,\ldots,x_ty_t,\emptyset,\ldots,\emptyset}[\mathcal A]-\mathbf {P}_{\beta}^{x_1y_1',\ldots,x_ty_t',\emptyset,\ldots,\emptyset}[\mathcal A]\Big|\leq \frac{\varepsilon}{3}.
\end{equation}
Using \eqref{eq: IIC critical}, we can find $z_1,\ldots,z_n\notin \Lambda_N$ such that
\begin{equation}
	\Big|\mathbf{P}_{\beta_c}^{x_1z_1,\ldots,x_tz_t,\emptyset,\ldots,\emptyset}[\mathcal A]-\mathbf P^{x_1\infty,\ldots,x_t\infty,\emptyset,\ldots,\emptyset}[\mathcal A]\Big|\leq \frac{\varepsilon}{3}.
\end{equation}
Finally, we use \cite[Theorem~2.3]{AizenmanDuminilSidoraviciusContinuityIsing2015} to argue that the map $\beta \mapsto \mathbf{P}_{\beta}^{x_1z_1,\ldots,x_tz_t,\emptyset,\ldots,\emptyset}[\mathcal A]$ is continuous at $\beta_c$. This yields the existence of $\beta_1\in [\beta(\varepsilon),\beta_c)$ such that for all $\beta\in [\beta_1,\beta_c]$,
\begin{equation}
	\Big|\mathbf{P}_{\beta}^{x_1z_1,\ldots,x_tz_t,\emptyset,\ldots,\emptyset}[\mathcal A]-\mathbf{P}_{\beta_c}^{x_1z_1,\ldots,x_tz_t,\emptyset,\ldots,\emptyset}[\mathcal A]\Big|\leq \frac{\varepsilon}{3}.
\end{equation}
Combining the last three displayed equations yields: for all $\beta \in [\beta_1,\beta_c]$, for all $y_1,\ldots, y_t\notin \Lambda_N$,
\begin{equation}
	\Big|\mathbf {P}_{\beta}^{x_1y_1,\ldots,x_ty_t,\emptyset,\ldots,\emptyset}[\mathcal A]-\mathbf P^{x_1\infty,\ldots,x_t\infty,\emptyset,\ldots,\emptyset}[\mathcal A]\Big|\leq \varepsilon.
\end{equation}
In particular,
\begin{equation}\label{eq: last eq IIC rcr}
	\limsup_{\beta \nearrow \beta_c}\frac{1}{\chi(\beta)^t}\sum_{y_1,\ldots,y_t\in \mathbb Z^d}\Big(\prod_{i=1}^t\langle \sigma_{x_i}\sigma_{y_i}\rangle_\beta\Big)\mathbf P_{\beta}^{x_1y_1,\ldots,x_ty_t,\emptyset,\ldots,\emptyset}[\mathcal A]\leq \varepsilon + \mathbf P^{x_1\infty,\ldots,x_t\infty,\emptyset,\ldots,\emptyset}[\mathcal A],
\end{equation}
where we used that $\lim_{\beta\nearrow\beta_c}\tfrac{\chi_N(\beta)}{\chi(\beta)}=0$ with $\chi_N(\beta):=\sum_{x\in\Lambda_N}\langle\sigma_0\sigma_x\rangle_{\beta}$. Since \eqref{eq: last eq IIC rcr} holds for any $\varepsilon>0$ this gives one inequality. The other inequality follows similarly. 
\end{proof}

The known properties of the spin model shed some light over the IIC of the random current. Below we state a (short) non-exhaustive list of properties one may easily derive in the case of the measure $\mathbf P^{0\infty,\emptyset}$.\begin{Prop}\label{prop: prop of rc IIC} Let $d\geq 4$. The measure $\mathbf P^{0\infty,\emptyset}$ constructed in Theorem~\textup{\ref{thm: IIC rcr}} satisfies the following properties:
\begin{enumerate}
\item[$(i)$] The cluster $\mathbf C(0)$ of $0$ under $\mathbf P^{0\infty,\emptyset}$ is one ended almost surely.
	\item[$(ii)$] For any $y\in \mathbb Z^d$, one has
	\begin{equation}\label{eq: connectivity in iic rcr}
		\mathbf P^{0\infty,\emptyset}[0\connect{\n_1+\n_2\:}y]=\langle \sigma_0\sigma_y\rangle_{\beta_c}.
	\end{equation}
	Moreover, there exist $c,C>0$ such that for every $y\in \mathbb Z^d$ with $|y|\geq 2$,
	\begin{equation}\label{eq: connectivity iic rcr d>3}
		\frac{C}{|y|^{d-2}}\geq \mathbf P^{0\infty,\emptyset}[0\connect{\n_1+\n_2\:}y]\geq \begin{cases}\displaystyle\frac{c}{|y|^{d-2}} &\text{ if }d\ge5,\\
    \displaystyle\frac{c}{|y|^{2}\log |y|}&\text{ if }d=4. 
    \end{cases}
	\end{equation}
\end{enumerate}
\end{Prop}
We postpone the proof of this result to Section \ref{section: proof of prop iic rc}.
The property $(ii)$ above suggests that, when $d\geq 4$, the IIC of the measure $\mathbf P^{0\infty,\emptyset}$ is a two-dimensional object. 
We can easily prove that the above result also holds when $d=3$ under the assumption\footnote{In fact, we even expect $\lim_{|x|\rightarrow \infty}|x|^{1+\varepsilon}\langle \sigma_0\sigma_x\rangle_{\beta_c}=0$ for some $\varepsilon>0$. This is related to the fact that the critical exponent $\eta$ of the two-point function is conjectured to satisfy $\eta>0$ in the $3D$ Ising model.} that $\lim_{|x|\rightarrow \infty}|x|\langle \sigma_0\sigma_x\rangle_{\beta_c}=0$.

\subsection{Construction of the IIC measure for the FK-Ising model}
Theorems \ref{thm: first construction IIC fk} and \ref{thm: subcritical approach to IIC FK Ising} are simple consequences of Proposition \ref{prop: coupling fk rcr} and Theorem \ref{thm: IIC rcr}. 
\begin{proof}[Proof of Theorem \textup{\ref{thm: first construction IIC fk}}] {\color{black}If $E\subset \mathbb E$ and $\eta,\eta'\in \{0,1\}^{E}$, we let $\eta\vee \eta':=(\eta_e\vee\eta_{e}')_{e\in E}$}. Let $\omega{\color{black}=(\omega_e)_{e\in \mathbb E}}$ be a Bernoulli percolation of parameter $1-\exp(-\beta_c)$, and write $\mathbb P_{1-\exp(-\beta_c)}$ for its law. Let $E\in \mathcal F_0$ be measurable in terms of edges in $\Lambda_n$. By Proposition \ref{prop: coupling fk rcr}, for all $x\in \mathbb Z^d$,
\begin{equation}\label{eq:proof iic fk1}
	\Big(\mathbf P_{\beta_c}^{0x}\otimes \mathbb P_{1-\exp(-\beta_c)}\Big)[(\mathds{1}_{\n_e>0})_{e\in \mathbb E}\vee \omega \in E]=\phi_{\beta_c}[E\:|\:0\connect{}x].
\end{equation}
Conditioning on the restriction $\omega_n$ of $\omega$ to the edges in $\Lambda_n$,
\begin{equation}\label{eq:proof iic fk2}
	\sum_{\omega_n \in \{0,1\}^{E(\Lambda_n)}}\mathbb P_{1-\exp(-\beta_c)}[\omega_n]\mathbf P_{\beta_c}^{0x}[(\mathds{1}_{\n_e>0})_{e\in {\color{black}E(\Lambda_n)}}\vee \omega_n \in E]=\phi_{\beta_c}[E\:|\:0\connect{}x].
\end{equation}
For a fixed configuration $\omega_n$, the event $\{(\mathds{1}_{\n_e>0})_{e\in {\color{black}E(\Lambda_n)}}\vee \omega_n \in E\}$ is local and hence by Theorem \ref{thm: IIC rcr} one has
\begin{align}
	\lim_{|x|\rightarrow \infty}\sum_{\omega_n \in \{0,1\}^{E(\Lambda_n)}}&\mathbb P_{1-\exp(-\beta_c)}[\omega_n]\mathbf P_{\beta_c}^{0x}[(\mathds{1}_{\n_e>0})_{e\in {\color{black}E(\Lambda_n)}}\vee \omega_n \in E]\notag
	\\&=\sum_{\omega_n \in \{0,1\}^{E(\Lambda_n)}}\mathbb P_{1-\exp(-\beta_c)}[\omega_n]\mathbf P^{0\infty}[(\mathds{1}_{\n_e>0})_{e\in {\color{black}E(\Lambda_n)}}\vee \omega_n \in E]\notag
	\\&=\Big(\mathbf P_{}^{0\infty}\otimes \mathbb P_{1-\exp(-\beta_c)}\Big)[(\mathds{1}_{\n_e>0})_{e\in \mathbb E}\vee \omega \in E].\label{eq:proof iic fk3}
\end{align}
Plugging \eqref{eq:proof iic fk3} in \eqref{eq:proof iic fk2}, it follows  that $\lim_{|x|\rightarrow \infty}\phi_{\beta_c}[E\:|\:0\connect{}x]$ exists and satisfies
\begin{equation}\label{eq: proof thm1 fk}
	\Big(\mathbf P_{}^{0\infty}\otimes \mathbb P_{1-\exp(-\beta_c)}\Big)[(\mathds{1}_{\n_e>0})_{e\in \mathbb E}\vee \omega \in E]=\lim_{|x|\rightarrow \infty}\phi_{\beta_c}[E\:|\:0\connect{}x].
\end{equation}
Since this holds for every $E\in \mathcal F_0$, the proof follows by defining $\phi^\infty$ to be the law of a percolation configuration sampled by $\mathbf P^{0\infty}$ and sprinkled by an independent Bernoulli percolation of parameter $1-\exp(-\beta_c)$.
\end{proof}
\begin{proof}[Proof of Theorem~\textup{\ref{thm: subcritical approach to IIC FK Ising}}] Using \eqref{eq: consequences of ESC} and Proposition \ref{prop: coupling fk rcr}, for every $\beta<\beta_c$, for every $E\in \mathcal F_0$,
\begin{align}
	\mathbb Q_\beta[E]&=\frac{1}{\chi(\beta)}\sum_{x\in \mathbb Z^d}\langle \sigma_0\sigma_x\rangle_\beta \phi_\beta[E\:|\:0\connect{}x]\notag
	\\&=\frac{1}{\chi(\beta)}\sum_{x\in \mathbb Z^d}\langle \sigma_0\sigma_x\rangle_\beta \Big(\mathbf P_{\beta}^{0x}\otimes \mathbb P_{1-\exp(-\beta)}\Big)[(\mathds{1}_{\n_e>0})_{e\in \mathbb E}\vee \omega \in E]\notag
	\\&=\Big(\frac{1}{\chi(\beta)}\sum_{x\in \mathbb Z^d}\langle \sigma_0\sigma_x\rangle_\beta\mathbf P_{\beta}^{0x}\Big)\otimes \mathbb P_{1-\exp(-\beta)}[(\mathds{1}_{\n_e>0})_{e\in \mathbb E}\vee \omega \in E]
\end{align}
Once again, if $E$ is measurable in terms of edges in $\Lambda_n$, for every $\omega_n\in \{0,1\}^{E(\Lambda_n)}$, Theorem \ref{thm: IIC rcr} implies the following convergence,
\begin{equation}
\lim_{\beta\nearrow\beta_c}\frac{1}{\chi(\beta)}\sum_{x\in \mathbb Z^d}\langle \sigma_0\sigma_x\rangle_\beta\mathbf P_{\beta}^{0x}[(\mathds{1}_{\n_e>0})_{e\in \mathbb E}\vee \omega_n \in E]=\mathbf P^{0\infty}[(\mathds{1}_{\n_e>0})_{e\in \mathbb E}\vee \omega_n \in E].
\end{equation}
Thus, as above,
\begin{equation}
	\lim_{\beta\nearrow \beta_c}\mathbb Q_\beta[E]=\Big(\mathbf P_{}^{0\infty}\otimes \mathbb P_{1-\exp(-\beta_c)}\Big)[(\mathds{1}_{\n_e>0})_{e\in \mathbb E}\vee \omega \in E]=\phi^\infty[E],
\end{equation}
where we used Theorem \ref{thm: first construction IIC fk} (or more specifically \eqref{eq: proof thm1 fk}) in the last equality.
\end{proof}

\section{Near-critical behaviour of the susceptibility when $d>4$: proof of Theorem \ref{thm: exact behaviour}}\label{section: susc}
We now turn to the study of the high-dimensional susceptibility of the Ising model. We begin by reminding a classical consequence of the switching lemma.

\subsection{Geometric interpretation of $\frac{\mathrm{d}\chi^{-1}}{\mathrm{d}\beta}$}
Set $J_{u,v}=\mathds{1}_{|u-v|_2=1}$ and let $|J|:=\sum_{x\in \mathbb Z^d}J_{0,x}=2d$. The following computation goes back to \cite{AizenmanGeometricAnalysis1982}.
We start by writing 
\begin{equation}\label{eq: proof thm 1 (1)}
	-\frac{\mathrm{d}\chi^{-1}(\beta)}{\mathrm{d}\beta}=\frac{1}{2}\frac{1}{\chi(\beta)^2}\sum_{y\in \mathbb Z^d}\sum_{u,v\in \mathbb Z^d}J_{u,v}\langle \sigma_0\sigma_y;\sigma_u\sigma_v\rangle_\beta.
\end{equation}
Now, notice that $\langle \sigma_0\sigma_y;\sigma_u\sigma_v\rangle_\beta=\langle \sigma_0\sigma_u\rangle_\beta\langle \sigma_v\sigma_y\rangle_\beta+\langle \sigma_0\sigma_v\rangle_\beta\langle \sigma_u\sigma_y\rangle_\beta+U_4^\beta(0,y,u,v)$, where $U_4^\beta(0,y,u,v)$ is Ursell's four-point function and is defined by
\begin{equation}
	U_4^\beta(0,y,u,v):=\langle \sigma_0\sigma_y\sigma_u\sigma_v\rangle_\beta-\langle \sigma_0\sigma_y\rangle_\beta\langle \sigma_u\sigma_v\rangle_\beta-\langle \sigma_0\sigma_u\rangle_\beta\langle \sigma_y\sigma_v\rangle_\beta-\langle \sigma_0\sigma_v\rangle_\beta\langle \sigma_y\sigma_u\rangle_\beta.
\end{equation}
A classical consequence of the switching lemma \cite{AizenmanGeometricAnalysis1982} is the following representation of $U_4^\beta$:
\begin{equation}\label{eq: representation ursell}
	U_4^\beta(0,y,u,v)=-2\langle \sigma_0\sigma_u\rangle_\beta\langle \sigma_v\sigma_y\rangle_\beta \mathbf P_\beta^{0u,vy}[\mathbf C_{\n_1+\n_2}(u)\cap \mathbf C_{\n_1+\n_2}(v)\neq \emptyset].
\end{equation}
Hence, using \eqref{eq: proof thm 1 (1)}, \eqref{eq: representation ursell}, and the symmetries of the model,
\begin{align}\notag
	-\frac{\mathrm{d}\chi^{-1}(\beta)}{\mathrm{d}\beta}&=|J|-\frac{1}{\chi(\beta)^2}\sum_{y,u,v\in \mathbb Z^d}J_{u,v}\langle \sigma_0\sigma_u\rangle_\beta\langle\sigma_v\sigma_y\rangle_\beta\mathbf P^{0u,vy}_\beta[\mathbf{C}_{\n_1+\n_2}(u)\cap \mathbf{C}_{\n_1+\n_2}(v)\neq \emptyset]
	\\\notag &=\frac{1}{\chi(\beta)^2}\sum_{y,u,v\in \mathbb Z^d}J_{u,v}\langle \sigma_0\sigma_u\rangle_\beta\langle\sigma_v\sigma_y\rangle_\beta\mathbf P^{0u,vy}_\beta[\mathbf{C}_{\n_1+\n_2}(u)\cap \mathbf{C}_{\n_1+\n_2}(v)=\emptyset]
	\\\notag &=\frac{1}{\chi(\beta)^2}\sum_{x,y\in \mathbb Z^d}\sum_{\substack{w\in \mathbb Z^d\\|w|_2=1}}\langle \sigma_0\sigma_x\rangle_\beta\langle\sigma_{w}\sigma_{y}\rangle_\beta\mathbf P^{0x,wy}_\beta[\mathbf{C}_{\n_1+\n_2}(0)\cap \mathbf{C}_{\n_1+\n_2}(w)=\emptyset]
	\\\label{eq: proof thm 1 (3)}
	&=\frac{2d}{\chi(\beta)^2}\sum_{x,y\in \mathbb Z^d}\langle \sigma_0\sigma_x\rangle_\beta\langle\sigma_{\mathbf e_1}\sigma_{y}\rangle_\beta\mathbf P^{0x,\mathbf{e}_1 y}_\beta[\mathbf{C}_{\n_1+\n_2}(0)\cap \mathbf{C}_{\n_1+\n_2}(\mathbf{e}_1)=\emptyset].
\end{align}
See Figure \ref{figure: event susceptibility} for an illustration. One of the contributions of this paper is a precise analysis of the probability appearing in \eqref{eq: proof thm 1 (3)}. The main difficulty to perform it lies in the lack of independence of $\mathbf{C}_{\n_1+\n_2}(0)$ and $\mathbf{C}_{\n_1+\n_2}(\mathbf{e}_1)$, which is usually helpful for the purpose of using the switching lemma.
To lighten the notations below, we will write for an event $\mathcal E$, 
\begin{equation}\label{eq: notation averaged proba}
	\overline{\mathbf{P}}_{\beta}^{(\ell)}[\mathcal E]:=\frac{1}{\chi(\beta)^2}\sum_{\substack{|x|\geq \ell\\|y|\geq \ell}}\langle \sigma_0\sigma_x\rangle_\beta\langle \sigma_{\mathbf{e}_1}\sigma_y\rangle_\beta \mathbf P^{0x,\mathbf{e}_1 y}_\beta[\mathcal E].
\end{equation}
This quantity corresponds to averaging the probability of occurrence of an event $\mathcal E$ with respect to the position of the sources.

\begin{figure}
	\begin{center}
		\includegraphics{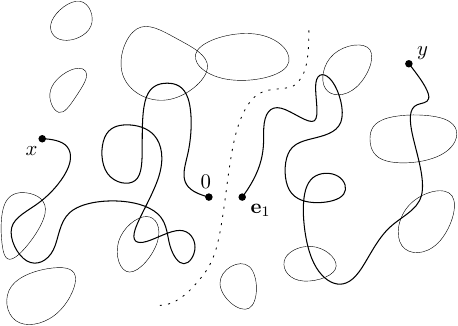}
		\caption{A configuration $\n$ which satisfies the event appearing in \eqref{eq: proof thm 1 (3)}. The dotted line highlights the fact that $0$ and $\mathbf{e}_1$ are not connected in the percolation configuration induced by $\n$.}
		\label{figure: event susceptibility}
	\end{center}
\end{figure}
\begin{Rem}\label{Rem: if x or y small then small}
	In \eqref{eq: proof thm 1 (3)}, the contribution of $x$ or $y$ in $\Lambda_\ell$ is bounded by 
\begin{equation}
	4d\frac{\chi_\ell(\beta)
}{\chi(\beta)},
\end{equation}
{\color{black}where $\chi_{\ell}(\beta):=\sum_{x\in \Lambda_{\ell}}\langle \sigma_0\sigma_x\rangle_{\beta}$.}
This can be made arbitrarily small by choosing $\beta$ close enough to $\beta_c$. As a result, as we approach $\beta_c$, it is enough to study $\overline{\mathbf{P}}_\beta^{(\ell)}[\mathbf{C}_{\n_1+\n_2}(0)\cap \mathbf{C}_{\n_1+\n_2}(\mathbf{e}_1)=\emptyset]$. This justifies the notation introduced in  \eqref{eq: notation averaged proba}.
\end{Rem}

\begin{Def}[Local event of avoidance] Let $k\geq 1$. We introduce the local event
\begin{equation}\label{eq:def Ak}
	\mathcal A_k:=\Big\{\n \in \Omega_{\mathbb Z^d}: \: \mathbf C_{\n_{|\Lambda_k}}(0)\cap \mathbf C_{\n_{|\Lambda_k}}(\mathbf e_1)=\emptyset\Big\},
\end{equation}
or, in words, the event that the clusters of $0$ and $\mathbf{e}_1$ in the restriction of $\n$ to $\Lambda_k$ do not intersect.
\end{Def}
The following lemma formalises the fact that the event appearing in \eqref{eq: proof thm 1 (3)} is essentially local around $0$ and $\mathbf{e}_1$. This is the main technical step of the proof.

\begin{Lem}\label{lem: localization step} {\color{black}Let $d>4$. }There exists $C=C(d),\eta>0$ such that, for all $\beta \leq \beta_c$, for all $k\geq 1$ and $\ell\geq 2k$,
\begin{equation}
	\left|\overline{\mathbf P}_{\beta}^{(\ell)}[\mathbf{C}_{\n_1+\n_2}(0)\cap \mathbf{C}_{\n_1+\n_2}(\mathbf{e}_1)=\emptyset]-\overline{\mathbf P}_{\beta}^{(\ell)}[\mathcal A_k]\right|\leq Ck^{-\eta}.
\end{equation}
\end{Lem}

Before proving Lemma \ref{lem: localization step}, we show how to conclude the proof of Theorems \ref{thm: exact behaviour} and \ref{thm: value of the constant}. 

\begin{proof}[Proof of Theorems \textup{\ref{thm: exact behaviour} and \textup{\ref{thm: value of the constant}}}]
Using Theorem \ref{thm: IIC rcr}, 
\begin{equation}\label{eq: proof s1}
	\lim_{\beta \nearrow \beta_c}\frac{1}{\chi(\beta)^2}\sum_{x,y\in \mathbb Z^d}\langle \sigma_0\sigma_x\rangle_\beta\langle \sigma_{\mathbf{e}_1}\sigma_y\rangle_\beta\mathbf P_\beta^{0x,\mathbf{e}_1y}[\mathcal A_k]=\mathbf P^{0\infty,\mathbf{e}_1\infty}[\mathcal A_k].
\end{equation}
Let $\varepsilon>0$. Combining \eqref{eq: proof thm 1 (3)}, Lemma \ref{lem: localization step} and \eqref{eq: proof s1}, if $\beta$ is sufficiently close to $\beta_c$ and $\ell\geq 2k$,
\begin{equation}
	-\frac{\mathrm{d}\chi^{-1}(\beta)}{\mathrm{d}\beta}\leq  4d\frac{\chi_{\ell}(\beta)}{\chi(\beta)}+2d\cdot \mathbf P^{0\infty,\mathbf{e}_1\infty}
[\mathcal A_k]+\frac{C}{k^\eta}+\varepsilon.
\end{equation}
Using Remark \ref{Rem: if x or y small then small}, we may choose $\beta$ even closer to $\beta_c$ to get
\begin{equation}
	-\frac{\mathrm{d}\chi^{-1}(\beta)}{\mathrm{d}\beta}\leq  2d\cdot \mathbf P^{0\infty,\mathbf{e}_1\infty}
[\mathcal A_k]+\frac{C}{k^\eta}+2\varepsilon.
\end{equation}
so that, integrating between $\beta$ and $\beta_c$ and taking $\beta$ to $\beta_c$ yields,
\begin{equation}\label{eq: final proof thm1 (1)}
	\limsup_{\beta\nearrow\beta_c} \big[\chi(\beta)(1-\beta/\beta_c)\big]^{-1}\leq (2d\beta_c)\cdot \mathbf P^{0\infty,\mathbf{e}_1\infty}
[\mathcal A_k]+\frac{C}{k^\eta}+2\varepsilon.
\end{equation}
Letting $k\rightarrow \infty$ and $\varepsilon \rightarrow 0$,
\begin{equation}
	\limsup_{\beta\nearrow \beta_c} \big[\chi(\beta)(1-\beta/\beta_c)\big]^{-1}\leq (2d\beta_c) \cdot\mathbf P^{0\infty,\mathbf{e}_1\infty}[\mathbf C_{\n_1+\n_2}(0)\cap \mathbf{C}_{\n_1+\n_2}(\mathbf{e}_1)=\emptyset].
\end{equation}
Similarly, if $\beta$ is close enough to $\beta_c$, $\overline{\mathbf P}_{\beta}^{(\ell)}[\mathcal A_k]\geq \mathbf P^{0\infty,\mathbf{e}_1\infty}[\mathcal A_k]-\tfrac{\varepsilon}{2d}$ so that
\begin{equation}
	-\frac{\mathrm{d}\chi^{-1}(\beta)}{\mathrm{d}\beta}\geq 2d\cdot\overline{\mathbf P}_{\beta}^{(\ell)}[\mathcal A_k]-\frac{C}{k^\eta}\geq  2d\cdot \mathbf P^{0\infty,\mathbf{e}_1\infty}
[\mathcal A_k]-\frac{C}{k^\eta}-\varepsilon.
\end{equation}
We deduce again that
\begin{equation}
	\liminf_{\beta \nearrow\beta_c} \big[\chi(\beta)(1-\beta/\beta_c)\big]^{-1}\geq (2d\beta_c) \cdot\mathbf P^{0\infty,\mathbf{e}_1\infty}[\mathbf C_{\n_1+\n_2}(0)\cap \mathbf{C}_{\n_1+\n_2}(\mathbf{e}_1)=\emptyset].
\end{equation}
The proof follows readily
\end{proof}
\begin{Rem}\label{rem: lower bound proba} In fact, the above proof also shows that 
\begin{equation}\label{eq: rem-2}
	\lim_{\beta\nearrow \beta_c}\left|\frac{\mathrm{d}\chi^{-1}(\beta)}{\mathrm{d}\beta}\right|=2d \cdot \mathbf P^{0\infty,\mathbf{e}_1\infty}[\mathbf C_{\n_1+\n_2}(0)\cap \mathbf{C}_{\n_1+\n_2}(\mathbf{e}_1)=\emptyset].
\end{equation}
Using \cite[(3.4)]{AizenmanGrahamRenormalizedCouplingSusceptibilityd=4-1983}, this gives
\begin{equation}\label{eq: rem-1}
	\mathbf P^{0\infty,\mathbf{e}_1\infty}[\mathbf C_{\n_1+\n_2}(0)\cap \mathbf{C}_{\n_1+\n_2}(\mathbf{e}_1)=\emptyset]\geq \frac{1}{1+(2d\beta_c) B(\beta_c)}.
\end{equation}
However, it is possible to obtain a slightly better bound, replacing $B(\beta_c)$ by $B^{(\mathbf{e}_1)}(\beta_c)$ (defined in \eqref{eq: def open bubble}). {\color{black} This is a minor improvement of \cite{AizenmanGrahamRenormalizedCouplingSusceptibilityd=4-1983} since (by the results of Sakai \cite{Sakai2007LaceExpIsing}) one has that $\lim_{d\rightarrow \infty}(2d\beta_c)B^{(\mathbf e_1)}(\beta_c)=0$, while $\lim_{d\rightarrow \infty}(2d\beta_c)B(\beta_c)=1$.} We briefly explain how to obtain this slightly better bound for sake of completeness. The lower bound on $\big|\tfrac{\mathrm{d}\chi^{-1}(\beta)}{\mathrm{d}\beta}\big|$ is based on an ``improved'' bound\footnote{{\color{black}The first bound on $|U^4_\beta|$ obtained in \cite{AizenmanGeometricAnalysis1982,FrohlichTriviality1982} comes with a $\langle \sigma_{x_1}\sigma_{u}\rangle_{\beta}\langle \sigma_{x_3}\sigma_{v}\rangle_{\beta}+\langle \sigma_{x_1}\sigma_{v}\rangle_{\beta}\langle \sigma_{x_3}\sigma_{u}\rangle_{\beta}$ instead of $\dfrac{\partial\langle \sigma_{x_1}\sigma_{x_3}\rangle_\beta}{\partial(\beta J_{u,v})}
$. The latter is smaller than the former by Lebowitz' inequality \cite{Lebowitz1974Inequ}.}} on $U_\beta^4$ \cite[(4.1')]{AizenmanGrahamRenormalizedCouplingSusceptibilityd=4-1983}: for $x_1,x_2,x_3,x_4\in \mathbb Z^d$,
\begin{multline}\label{eq: rem0}
	|U_4^\beta(x_1,x_2,x_3,x_4)|\leq \sum_{u,v\in \mathbb Z^d}J_{u,v}\langle \sigma_{x_4}\sigma_v\rangle_\beta\langle \sigma_{x_2}\sigma_v\rangle_\beta (\beta J_{u,v})\dfrac{\partial\langle \sigma_{x_1}\sigma_{x_3}\rangle_\beta}{\partial(\beta J_{u,v})}
	\\+\langle \sigma_{x_1}\sigma_{x_2}\rangle_\beta\langle \sigma_{x_1}\sigma_{x_3}\rangle_\beta \langle \sigma_{x_1}\sigma_{x_4}\rangle_\beta + \langle \sigma_{x_3}\sigma_{x_1}\rangle_\beta\langle \sigma_{x_3}\sigma_{x_2}\rangle_\beta \langle \sigma_{x_3}\sigma_{x_4}\rangle_\beta.
\end{multline}
As in \eqref{eq: proof thm 1 (1)} and below, and using Lebowitz' inequality $U_4^\beta\leq 0$ \cite{Lebowitz1974Inequ}, we write
\begin{equation}\label{eq: rem1}
	\left|\frac{\mathrm{d}\chi^{-1}(\beta)}{\mathrm{d}\beta}\right|=|J|-\frac{1}{2}\frac{1}{\chi(\beta)^2}\sum_{x,y,z\in \mathbb Z^d}J_{y,z}|U_4^\beta(0,y,x,z)|.
\end{equation}
Injecting \eqref{eq: rem0} in \eqref{eq: rem1},
\begin{align}
	\left|\frac{\mathrm{d}\chi^{-1}(\beta)}{\mathrm{d}\beta}\right|&\geq |J|-\frac{1}{2}\frac{1}{\chi(\beta)^2}\sum_{x,y,z,u,v\in \mathbb Z^d}J_{y,z}\langle \sigma_z\sigma_v\rangle_\beta\langle \sigma_y\sigma_v\rangle_\beta(\beta J_{u,v})\dfrac{\partial\langle \sigma_0\sigma_{x}\rangle_\beta}{\partial(\beta J_{u,v})} \notag
	\\&\qquad -\frac{1}{2}\frac{1}{\chi(\beta)^2}\sum_{x,y,z\in \mathbb Z^d}J_{y,z}\langle \sigma_0\sigma_x\rangle_\beta\langle\sigma_0\sigma_y\rangle_\beta\langle \sigma_0\sigma_z\rangle_\beta \notag
	\\&\qquad -\frac{1}{2}\frac{1}{\chi(\beta)^2}\sum_{x,y,z\in \mathbb Z^d}J_{y,z}\langle \sigma_x\sigma_0\rangle_\beta\langle\sigma_x\sigma_y\rangle_\beta\langle \sigma_x\sigma_z\rangle_\beta \notag
	\\&=:|J|-(I)-(II)-(III).\label{eq: rem3}
\end{align}
Now, using translation invariance\footnote{In \cite{AizenmanGrahamRenormalizedCouplingSusceptibilityd=4-1983}, the authors apply the Cauchy--Schwarz inequality and get $B(\beta)$ instead of $B^{(\mathbf{e}_1)}(\beta)$.}
\begin{align}
	(I)&=\frac{1}{2}\frac{1}{\chi(\beta)^2}\sum_{x,u,v\in \mathbb Z^d}(\beta J_{u,v})\dfrac{\partial\langle \sigma_0\sigma_{x}\rangle_\beta}{\partial(\beta J_{u,v})}\Big(\sum_{y\in \mathbb Z^d}\sum_{|w|_2=1}\langle \sigma_y\sigma_v\rangle_\beta\langle\sigma_{y+w}\sigma_v\rangle_\beta\Big)
	\\&=(2d\beta) B^{(\mathbf{e}_1)}(\beta)\left|\frac{\mathrm{d}\chi^{-1}(\beta)}{\mathrm{d}\beta}\right|,\label{eq: rem4}
\end{align}
where we recall that $B^{(\mathbf{e}_1)}(\beta)=\sum_{x\in \mathbb Z^d}\langle \sigma_0\sigma_x\rangle_\beta\langle \sigma_{\mathbf{e}_1}\sigma_x\rangle_\beta$.
Moreover, using the Cauchy--Schwarz inequality
{\color{black}
\begin{align}
	\sum_{x,y,z\in \mathbb Z^d}J_{y,z}\langle \sigma_0\sigma_x\rangle_\beta\langle\sigma_0\sigma_y\rangle_\beta\langle \sigma_0\sigma_z\rangle_\beta&=\chi(\beta)\sum_{y,z\in \mathbb Z^d}J_{y,z}\langle\sigma_0\sigma_y\rangle_\beta\langle \sigma_0\sigma_z\rangle_\beta\notag
	\\&\leq \chi(\beta)\Big(\sum_{y,z\in \mathbb Z^d}J_{y,z}\langle \sigma_0\sigma_y\rangle_{\beta}^2\Big)^{1/2}\Big(\sum_{y,z\in \mathbb Z^d}J_{y,z}\langle \sigma_0\sigma_z\rangle_{\beta}^2\Big)^{1/2}\notag
	\\&=\chi(\beta)|J|B(\beta).
\end{align}
A similar computation can be made to bound $(III)$. This gives
}
\begin{equation}\label{eq: rem5}
	(II)+(III)\leq \frac{|J|B(\beta)}{\chi(\beta)}.
\end{equation}
Plugging \eqref{eq: rem4} and \eqref{eq: rem5} in \eqref{eq: rem3}, {\color{black} and using that $|J|=2d$} gives
\begin{equation}
	\left|\frac{\mathrm{d}\chi^{-1}(\beta)}{\mathrm{d}\beta}\right|\geq \frac{2d}{1+(2d\beta_c)B^{(\mathbf{e}_1)}(\beta_c)}\Big(1-\frac{B(\beta_c)}{\chi(\beta)}\Big).
\end{equation}
Since $B(\beta_c)<\infty$ by \eqref{eq: IRB}, one gets
\begin{equation}
	\lim_{\beta\nearrow \beta_c}\left|\frac{\mathrm{d}\chi^{-1}(\beta)}{\mathrm{d}\beta}\right|\geq \frac{2d}{1+(2d\beta_c)B^{(\mathbf{e}_1)}(\beta_c)},
\end{equation}
which immediately yields the improvement on \eqref{eq: rem-1} by \eqref{eq: rem-2}.
\end{Rem}
\subsection{Proof of Lemma \ref{lem: localization step}}
The rest of the section is devoted to the proof of the \emph{localization} result of Lemma \ref{lem: localization step}. We will use the notion of
 \emph{backbone} of a current introduced in \cite[Section~4]{AizenmanBarskyFernandezSharpnessIsing1987} (see also \cite{AizenmanDuminilTassionWarzelEmergentPlanarity2019,PanisTriviality2023}).
A current $\n$ on a $\Lambda\subset \mathbb Z^d$ finite with sources $\sn=\{x,y\}$ has at least one path from $x$ to $y$ in the percolation configuration induced by $\n$. The backbone of $\n$ is an exploration of one such path. We fix an arbitrary ordering $\prec$ of the edges of $E(\Lambda)$.

\begin{Def}
\label{def:backbone}
Let $\n\in \Omega_{\Lambda}$ with $\sn=\lbrace x,y\rbrace$. The \emph{backbone}
of $\n$, which we denote by $\Gamma(\n)$, is the unique oriented and edge self-avoiding path
from $x$ to $y$, supported on edges with $\n$ odd, which is minimal for $\prec$.
The backbone $\Gamma(\n)=\{x_ix_{i+1}:  0\leq i < k\}$ is obtained via the following procedure:
\begin{enumerate}
    \item[$(i)$] Set $x_0=x$. The first edge $x_0x_1$ of $\Gamma(\n)$ is the earliest of all the edges $e$ emerging from $x_0$ for which $\n_e$ is odd.
    All edges $e\ni x_0$ satisfying $e\prec x_0x_1$ and $e\neq x_0x_1$ are explored and have $\n_e$ even.
    \item[$(ii)$] Each edge $x_ix_{i+1}$, $i\ge 1$, is the first of all edges $e$ emerging from $x_i$ that have not been explored previously, and for which $\n_e$ is odd.
    \item[$(iii)$] The exploration stops when it reaches $y=x_k$ for the first time.
\end{enumerate}
Let $\overline{\Gamma}(\n)$ denote the set of explored edges, i.e. $\Gamma(\n)$ together with all explored even edges.
\end{Def}

By definition, $\overline \Gamma(\n)$ is determined by $\Gamma(\n)$ and the ordering $\prec$,
and similarly $\Gamma(\n)$ can be recovered from $\overline\Gamma(\n)$ by choosing odd edges in $\overline\Gamma(\n)$ according to $\prec$.
The current $\n\setminus \overline \Gamma(\n)$, defined to be the restriction of $\n$ to the complement of $\overline{\Gamma}(\n)$, is sourceless. Moreover,  we can write
\begin{equation}\label{eq: backbone exp}
    \langle \sigma_x\sigma_y\rangle_{\Lambda,\beta}=\sum_{\gamma :x \rightarrow y}\rho_\Lambda(\gamma),
\end{equation}
where for a path $\gamma:x \rightarrow y$,
\begin{equation}
    \rho_\Lambda(\gamma):=\langle \sigma_x\sigma_y\rangle_\beta\mathbf P_{\Lambda,\beta}^{xy}[\Gamma(\n)=\gamma].
\end{equation}
When $\Lambda=\mathbb Z^d$, we write $\rho(\gamma)=\rho_{\Lambda}(\gamma)$.
We will use the following useful property of backbones, often referred to as the \emph{chain rule} (see \cite{AizenmanBarskyFernandezSharpnessIsing1987}).
\begin{Prop}[Chain rule for the backbone]\label{prop: chain rule}  Let $\Lambda \subset \mathbb Z^d$. Let $x,y,u,v\in \Lambda$. Then,
\begin{equation}
    \mathbf{P}^{xy}_{\Lambda,\beta}[\Gamma(\n) \textup{ passes through }u \textup{ first and then through }v]\leq \frac{\langle \sigma_x\sigma_u\rangle_{\Lambda,\beta}\langle \sigma_u\sigma_v\rangle_{\Lambda,\beta}\langle \sigma_v\sigma_y\rangle_{\Lambda,\beta}}{\langle \sigma_x\sigma_y\rangle_{\Lambda,\beta}}.
\end{equation}
In particular, if $u=v$, 
\begin{equation}
	\mathbf{P}^{xy}_{\Lambda,\beta}[\Gamma(\n) \textup{ passes through }u]\leq \frac{\langle \sigma_x\sigma_u\rangle_{\Lambda,\beta}\langle \sigma_u\sigma_y\rangle_{\Lambda,\beta}}{\langle \sigma_x\sigma_y\rangle_{\Lambda,\beta}}.
\end{equation}
\end{Prop}
{\color{black}We now turn to the proof of Lemma \ref{lem: localization step}. Recall that
\begin{equation}
	\mathcal A_k=\Big\{\n \in \Omega_{\mathbb Z^d}, \: \mathbf C_{\n_{|\Lambda_k}}(0)\cap \mathbf C_{\n_{|\Lambda_k}}(\mathbf e_1)=\emptyset\Big\}.
\end{equation}}
\begin{proof}[Proof of Lemma~\textup{\ref{lem: localization step}}] Clearly $\{ \mathbf{C}_{\n_1+\n_2}(0)\cap \mathbf{C}_{\n_1+\n_2}(\mathbf{e}_1)=\emptyset\}\subset \mathcal A_k$. It then suffices to analyse the event $\mathcal B:=\mathcal A_k\cap \{\mathbf{C}_{\n_1+\n_2}(0)\cap \mathbf{C}_{\n_1+\n_2}(\mathbf{e}_1)\neq\emptyset\}$. Note that if the two clusters intersect, there must be a path from $\mathbf{e}_1$ to $\Gamma(\n_1)$ that does not use any edge of $\overline{\Gamma}(\n_1)$, except maybe for the last edge:
\begin{equation}\label{eq: inclusion lem local}
	\{\mathbf{C}_{\n_1+\n_2}(0)\cap \mathbf{C}_{\n_1+\n_2}(\mathbf{e}_1)\neq\emptyset\}\subset \bigcup_{\substack{u\in \mathbb Z^d\\v\in \mathbb Z^d\\u\sim v}} \{u\in \Gamma(\n_1)\}\cap \{\mathbf{e}_1\connect{\n_1+\n_2\setminus \overline \Gamma(\n_1)\:}v\}\cap \{(\n_1+\n_2)_{u,v}>0\}.
\end{equation}
Let $\mathcal B_{u,v}:=\mathcal A_k\cap \{u\in \Gamma(\n_1)\}\cap \{\mathbf{e}_1\connect{\n_1+\n_2\setminus \overline \Gamma(\n_1)\:}v\}\cap \{(\n_1+\n_2)_{u,v}>0\}$. Let $m=k^\delta$ for some small $\delta>0$ to be fixed.

$\bullet$ We first look at the case $u\notin \Lambda_m$, see Figure \ref{figure:localisation1} for an illustration. Exploring $\Gamma(\n_1)$, and forgetting about $ \mathcal B\cap \{(\n_1+\n_2)_{u,v}>0\}$, we get\footnote{A similar argument is used in the proof of \cite[Lemma~6.13]{PanisTriviality2023}.}
\begin{align}
	\langle \sigma_0\sigma_x\rangle_{\beta}\langle \sigma_{\mathbf{e}_1}\sigma_y\rangle_\beta\mathbf P_\beta^{0x,\mathbf{e}_1 y}[\mathcal B_{u,v}]&\leq\sum_{\substack{\gamma:0\rightarrow x\\ \gamma\ni u}}\rho(\gamma)\langle \sigma_{\mathbf{e}_1}\sigma_y\rangle_\beta\mathbf{P}^{\emptyset,\mathbf{e}_1y}_{\overline\gamma^c,\mathbb Z^d,\beta}[\mathbf{e}_1\connect{\n_1+\n_2\setminus \overline \gamma\:}v]\notag
	\\&=\sum_{\substack{\gamma:0\rightarrow x\\ \gamma\ni u}}\rho(\gamma)\langle \sigma_{\mathbf{e}_1}\sigma_v\rangle_{\overline{\gamma}^c,\beta}\langle \sigma_v\sigma_y\rangle_\beta\notag
	\\&\leq \langle \sigma_{\mathbf{e}_1}\sigma_v\rangle_\beta\langle \sigma_v\sigma_y\rangle_\beta\sum_{\substack{\gamma_1:0\rightarrow u\\\gamma_2:u\rightarrow x}}\rho(\gamma_1\circ \gamma_2)\notag
	\\&\leq \langle \sigma_{\mathbf{e}_1}\sigma_v\rangle_\beta\langle \sigma_v\sigma_y\rangle_\beta\langle \sigma_{0}\sigma_u\rangle_\beta\langle \sigma_u\sigma_x\rangle_\beta \label{eq: bound u not in lambda m},
\end{align}
where we used the switching lemma on the second line, Griffiths' inequality on the third line, and Proposition \ref{prop: chain rule} on the last line. As a result, using the infrared bound \eqref{eq: IRB}, there exists $C_1=C_1(d)>0$ such that
\begin{equation}\label{eq:bound u far away}
	\frac{1}{\chi(\beta)^2}\sum_{|x|,|y|\geq \ell} \sum_{u\notin \Lambda_m}\sum_{v\sim u}\langle \sigma_0\sigma_x\rangle_{\beta}\langle \sigma_{\mathbf{e}_1}\sigma_y\rangle_\beta\mathbf P_\beta^{0x,\mathbf{e}_1 y}[\mathcal B_{u,v}]\leq \sum_{u\notin \Lambda_m}\sum_{v\sim u}\langle \sigma_0\sigma_u\rangle_\beta\langle \sigma_{\mathbf{e_1}}\sigma_u\rangle_\beta\leq \frac{C_1}{m^{d-4}}.
\end{equation}

\begin{figure}
	\begin{center}
		\includegraphics{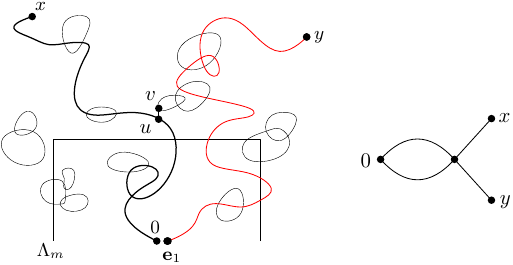}
		\caption{On the left, an illustration of a configuration realizing $\mathcal B_{u,v}$ for $u\notin \Lambda_m$. The backbone of $\n_1$ (resp. $\n_2$) is the black (resp. red) bold line. A string of loops connects $\Gamma(\n_2)$ to $\Gamma(\n_1)$ in $\n_1+\n_2$ without using any edges of $\overline \Gamma(\n_1)$. On the right, a diagrammatic representation of the bound obtained in \eqref{eq: bound u not in lambda m}.}
		\label{figure:localisation1}
	\end{center}
\end{figure}

$\bullet$ We now turn to the case $u\in \Lambda_m$. We introduce an additional intermediate scale $M=m^A\leq k$ with $A>0$ large enough\footnote{Since $M\leq k$, one has $A\delta\leq 1$. Hence, how large $A$ can be taken will depend of how small $\delta$ is chosen.} to be fixed.  {\color{black}We now introduce the following events: 
\begin{enumerate}
	\item[-]$\mathcal G_1:$ one of the backbones $\Gamma(\n_1)$ or $\Gamma(\n_2)$ does two successive crossings of $\textup{Ann}(M,k)$,
	\item[-]$\mathcal G_2:$ the event $\mathcal G_1$ does not occur and there exist $s\in \partial \Lambda_M$ and $t\in \partial \Lambda_m$ such that $s \connect{\n_1+\n_2\setminus(\overline \Gamma(\n_1)\cup\overline\Gamma(\n_2))\:}t$.
\end{enumerate}

We now prove the following claim. See Figure \ref{fig:figure loc2} for an illustration.
\begin{Claim}\label{claim} Let $u\in \Lambda_m$ and $v\sim u$. If the event $\mathcal B_{u,v}$ occurs, then $\mathcal G_1$ or $\mathcal G_2$ must occur.
\end{Claim}
\begin{proof}[Proof of Claim \textup{\ref{claim}}] Assume that the event $\mathcal B_{u,v}$ occurs but not $\mathcal G_1$. We prove that the event $\mathcal G_2$ must occur. 

First, observe that since the event $\mathcal G_1$ does not occur, if $w\in \Lambda_M$ is visited by $\Gamma(\n_i)$ ($i=1,2$), then the portion of $\Gamma(\n_i)$ from $0$ to $w$ is entirely contained in $\Lambda_k$. In particular, the portion of $\Gamma(\n_1)$ from $0$ to $u$ is included in $\Lambda_k$.

 Since the event $\mathcal B_{u,v}$ occurs, there is an open path from $\mathbf{e}_1$ to $v$ in $\n_1+\n_2\setminus \overline\Gamma(\n_1)$. We call $\gamma$ such an open path. By the first observation, the path $\gamma$ must exit $\Lambda_M$ before reaching $v$. Indeed, if not, since $\{u\in \Gamma(\n_1)\}\cap \{(\n_1+\n_2)_{u,v}>0\}$ also occurs, it would create a connection from $\mathbf{e_1}$ to $0$ which is entirely contained in $\Lambda_k$, and would contradict the occurrence of $\mathcal A_k$. 
 
 Call $s$ the last vertex of $\partial \Lambda_M$ visited by $\gamma$ before reaching $v$, and $\tilde{\gamma}$ the portion of $\gamma$ between $s$ and $v$. We now claim that $\tilde{\gamma}$--- which by definition does not use any edge of $\overline{\Gamma}(\n_1)$--- cannot use any edge of $\overline{\Gamma}(\n_2)$. Assume by contradiction that it goes through some edge $e=ab\in \overline{\Gamma}(\n_2)$ with $b\in \Gamma(\n_2)$. Concatenating the portion of $\Gamma(\n_2)$ from $\mathbf{e}_1$ to $b$ (which is entirely contained in $\Lambda_k$ by the first observation), and the portion of $\tilde{\gamma}$ from $b$ to $v$ creates a connection from $\mathbf{e}_1$ to $0$ in $\Lambda_k$: this contradicts this occurrence of $\mathcal A_k$ and concludes the proof. 
\end{proof}
}

We begin with a bound of $\mathcal G_1$. A similar bound was already used in \cite[(4.21)]{AizenmanDuminilTriviality2021}. {\color{black}Let us write for $i=1,2$, $\mathcal G_1(i):=\{\Gamma(\n_i) \textup{ does two successive crossings of Ann}(M,k)\}$, so that $\mathcal G_1\subset \mathcal G_1(1)\cup \mathcal G_1(2)$.} By the chain rule of Proposition \ref{prop: chain rule},
\begin{align}
	\langle \sigma_0\sigma_x\rangle_{\beta}\langle \sigma_{\mathbf{e}_1}\sigma_y\rangle_\beta\mathbf P_\beta^{0x,\mathbf{e}_1 y}[\mathcal G_1{\color{black}(1)}]&=\langle \sigma_0\sigma_x\rangle_{\beta}\langle \sigma_{\mathbf{e}_1}\sigma_y\rangle_\beta \mathbf P_\beta^{{\color{black}0x}}[\mathcal G_1{\color{black}(1)}]\notag
	\\&\leq {\color{black}\langle \sigma_{\mathbf{e}_1}\sigma_y\rangle_\beta} \sum_{\substack{a\in \partial \Lambda_k\\b \in \partial \Lambda_M}}\langle \sigma_{{\color{black}0}}\sigma_a\rangle_\beta\langle \sigma_a\sigma_b\rangle_\beta\langle \sigma_b\sigma_{{\color{black}x}}\rangle_\beta.\label{eq:bound g1(1)}
\end{align}
Using once again the infrared bound \eqref{eq: IRB}, there exists $C_2=C_2(d)>0$ such that,
\begin{equation}\label{eq: bound F_1}
	\overline{\mathbf P}^{(\ell)}_\beta[\mathcal G_1{\color{black}(1)}]\leq \sum_{\substack{a\in \partial \Lambda_k\\b \in \partial \Lambda_M}}\langle \sigma_{{\color{black}0}}\sigma_a\rangle_\beta\langle \sigma_a\sigma_b\rangle_\beta \leq C_2\frac{k^{d-1}M^{d-1}}{k^{2d-4}}=C_2\frac{M^{d-1}}{k^{d-3}}=C_2\frac{k^{(d-1)\delta A}}{k^{d-3}}.
\end{equation}
{\color{black}A similar computation allows to bound $\overline{\mathbf P}^{(\ell)}_\beta[\mathcal G_1{\color{black}(2)}]$ so that we obtain $C_3=C_3(d)>0$ such that
\begin{equation}\label{eq: bound F_1}
	\overline{\mathbf P}^{(\ell)}_\beta[\mathcal G_1]\leq  C_3\frac{k^{(d-1)\delta A}}{k^{d-3}}.
\end{equation}
}
\begin{figure}
	\begin{center}
		\includegraphics{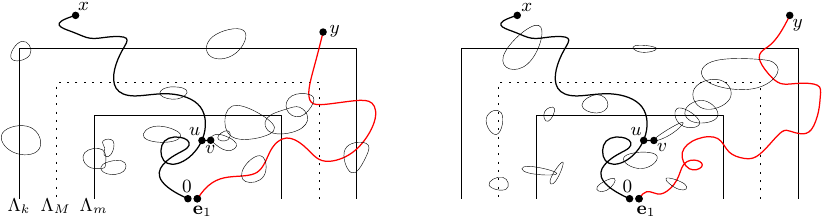}
		\caption{An illustration of the events $\mathcal G_1$ and $\mathcal G_2$. The backbone of $\n_1$ (resp. $\n_2$) is the black (resp. red) bold line. On the left, $\Gamma(\n_2)$ crosses $\text{Ann}(M,k)$ twice and a string of loops creates a connection between the clusters of $0$ and $\mathbf{e}_1$. On the right, {\color{black} $\Gamma(\n_1)$ and }$\Gamma(\n_2)$ only cross $\text{Ann}(M,k)$ once. This forces the existence of a crossing of $\textup{Ann}(m,M)$ in $\n_1+\n_2\setminus \overline\Gamma(\n_1)\cup\overline\Gamma(\n_2)$.}
		\label{fig:figure loc2}
	\end{center}
\end{figure}
We turn to the bound of $\mathcal{G}_2$. Write,
\begin{align}\label{eq: proof localisation f2 (1°}
	\langle \sigma_0\sigma_x\rangle_{\beta}\langle \sigma_{\mathbf{e}_1}\sigma_y\rangle_\beta\mathbf P_\beta^{0x,\mathbf{e}_1 y}[\mathcal G_2]
	&\leq \sum_{\substack{s\in \partial \Lambda_M\\ t \in \partial \Lambda_m}}\langle \sigma_0\sigma_x\rangle_{\beta}\langle \sigma_{\mathbf{e}_1}\sigma_y\rangle_\beta\mathbf P_\beta^{0x,\mathbf{e}_1y}[s \connect{\n_1+\n_2\setminus(\overline \Gamma(\n_1)\cup\overline\Gamma(\n_2))\:}t].
\end{align}
{\color{black}Let ${C(s,t) :=\{s \connect{\n_1+\n_2\setminus(\overline \Gamma(\n_1)\cup\overline\Gamma(\n_2))\:}t\}}$. The end of the proof relies on the following claim.
\begin{Claim}\label{claim2} One has
\begin{equation}
	\mathbf P_\beta^{0x,\mathbf{e}_1y}[C(s,t)]\leq \langle \sigma_s\sigma_t\rangle_\beta^2.
\end{equation}
\end{Claim}

Recall that for $\Lambda \subset \mathbb Z^d$, $\Omega_\Lambda$ is the set of currents on $E(\Lambda)$. Let us introduce some notations for the partition functions of the random current measures: if $\Lambda,\Lambda'$ are finite subsets of $\mathbb Z^d$, $A\subset \Lambda$ and $B\subset \Lambda'$, $\mathcal E_1\subset \Omega_{\Lambda}$ and $\mathcal E_2\subset \Omega_{\Lambda}\times \Omega_{\Lambda'}$, we let
\begin{align}
Z_{\Lambda,\beta}^A[\mathcal E_1]&=\sum_{\n\in \Omega_\Lambda:\sn=A}w(\n)\mathds{1}\{\n\in \mathcal E_1\},
\\ 
	Z_{\Lambda,\Lambda',\beta}^{A,B}[\mathcal E_2]&=\sum_{\substack{\n_1\in \Omega_{\Lambda}: \sn_1=A\\\n_2\in \Omega_{\Lambda'}:\sn_2=B}}w(\n_1)w(\n_2)\mathds{1}\{(\n_1,\n_2)\in \mathcal E\}.
\end{align}
}
\begin{proof}[Proof of Claim \textup{\ref{claim2}}] We go to partition functions\footnote{For full disclosure, this step requires working in a finite volume $\Lambda$ and then taking the limit $\Lambda\nearrow \mathbb Z^d$. We omit this detail here as it does not affect the argument.} and write
\begin{align}
	Z_\beta^{0x,\mathbf{e}_1y}[C(s,t)]=\sum_{\substack{\gamma_1:0\rightarrow x\\\gamma_2:\mathbf{e}_1\rightarrow y}}Z_\beta^{0x}[\Gamma(\n_1)=\gamma_1]Z_\beta^{\mathbf{e}_1y}[\Gamma(\n_2)=\gamma_2]\mathbf P^{\emptyset,\emptyset}_{\overline \gamma_1^c,\overline \gamma_2^c,\beta}[C(s,t)].
\end{align}
Now, write for $\n_1\in \Omega_{\overline\gamma_1^c}$ and $\n_2\in \Omega_{\overline\gamma_2^c}$
\begin{equation}\notag
	\n_1=\mathbf{k}_1+\m_1, \qquad \mathbf{k}_1(e)=\n_1(e)\mathds{1}\{e\in \overline\gamma_2\} \qquad (e\in \mathbb E),
\end{equation}
\begin{equation}\notag
	\n_2=\mathbf{k}_2+\m_2, \qquad \mathbf{k}_2(e)=\n_2(e)\mathds{1}\{e\in \overline\gamma_1\} \qquad (e\in \mathbb E),
\end{equation}
so that $\m_1$ and $\m_2$ are both currents on $(\overline \gamma_1\cup \overline \gamma_2)^c$. Then,
\begin{align}
	Z^{\emptyset,\emptyset}_{\overline \gamma_1^c,\overline \gamma_2^c,\beta}[C(s,t)]&=\sum_{\mathbf{k}_1,\mathbf{k}_2}w_\beta(\mathbf{k}_1)w_\beta(\mathbf{k}_2)\sum_{\substack{\partial \m_1=\partial \mathbf{k}_1\\ \partial \m_2=\partial \mathbf{k}_2}}w_\beta(\m_1)w_\beta(\m_2)\mathds{1}\{C(s,t)\}\notag
	\\&=\sum_{\mathbf{k}_1,\mathbf{k}_2}w_\beta(\mathbf{k}_1)w_\beta(\mathbf{k}_2)\sum_{\substack{\partial \m_1=\partial \mathbf{k}_1\Delta \{s,t\}\\ \partial \m_2=\partial \mathbf{k}_2\Delta \{s,t\}}}w_\beta(\m_1)w_\beta(\m_2)\mathds{1}\{C(s,t)\}\notag
	\\&=Z^{st,st}_{\overline \gamma_1^c,\overline \gamma_2^c,\beta}[C(s,t)]\leq Z^{st,st}_{\overline \gamma_1^c,\overline \gamma_2^c,\beta},
\end{align}
where on the second line we used a slightly more general version of the switching lemma, called the \emph{switching principle}, see \cite[Lemma~2.1]{AizenmanDuminilTassionWarzelEmergentPlanarity2019}. As a result,
\begin{equation}\label{eq: bound C(s,t)}
	\mathbf P^{\emptyset,\emptyset}_{\overline \gamma_1^c,\overline \gamma_2^c,\beta}[C(s,t)]\leq \langle \sigma_s\sigma_t\rangle_{\overline\gamma_1^c,\beta}\langle \sigma_s\sigma_t\rangle_{\overline\gamma_2^c,\beta}.
	\end{equation}
Collecting the above work and using Griffiths' inequality yields
\begin{equation}\label{eq: proof localisation f2 (2)}
	\mathbf P_\beta^{0x,\mathbf{e}_1y}[C(s,t)]\leq \sum_{\substack{\gamma_1:0\rightarrow x\\\gamma_2:\mathbf{e}_1\rightarrow y}}\mathbf P_\beta^{0x}[\Gamma(\n_1)=\gamma_1]\mathbf P_\beta^{\mathbf{e}_1y}[\Gamma(\n_2)=\gamma_2]\langle \sigma_s\sigma_t\rangle_{\overline\gamma_1^c,\beta}\langle \sigma_s\sigma_t\rangle_{\overline\gamma_2^c,\beta}\leq \langle \sigma_s\sigma_t\rangle_\beta^2,
\end{equation}
{\color{black}which concludes the proof of the claim.}
\end{proof}
{\color{black} Using Claim \ref{claim2} in \eqref{eq: proof localisation f2 (1°}}, averaging over $x$ and $y$, {\color{black}and using the infrared bound \eqref{eq: IRB}} gives $C_4=C_4(d)>0$ such that
\begin{equation}\label{eq: bound f2}
	\overline{\mathbf P}^{(\ell)}_\beta[\mathcal G_2]\leq \sum_{\substack{s\in \partial \Lambda_M\\ t \in \partial \Lambda_m}}\langle \sigma_s\sigma_t\rangle_\beta^2\leq C_3\frac{M^{d-1}m^{d-1}}{M^{2d-4}}=C_4\frac{k^{\delta(d-1)}}{k^{A\delta(d-3)}}.
\end{equation}
{\color{black}Combining \eqref{eq:bound u far away}}, \eqref{eq: bound F_1}, and \eqref{eq: bound f2} gives $C_5=C_5(d),C_6=C_6(d)>0$ such that
\begin{align}
	\frac{1}{\chi(\beta)^2}\sum_{|x|,|y|\geq \ell} \sum_{u\in \Lambda_m}\sum_{v\sim u}\langle \sigma_0\sigma_x\rangle_{\beta}\langle \sigma_{\mathbf{e}_1}\sigma_y\rangle_\beta &\mathbf P_\beta^{0x,\mathbf{e}_1 y}[\mathcal B_{u,v}]\notag\\&\leq \frac{C_5}{k^{\delta(d-4)}}+C_5m^d\left(\frac{k^{(d-1)\delta A}}{k^{d-3}}+\frac{k^{\delta(d-1)}}{k^{A\delta(d-3)}}\right)\notag
	\\&\leq C_6k^{-\eta},\label{eq:bound k to the eta}
\end{align}
for some $\eta>0$ sufficiently small, {\color{black}where we chose $A=3$ and $\delta<\tfrac{d-3}{4d-3}$}. The proof follows readily.
\end{proof}

\subsection{Proof of Theorem \ref{prop:uniform lower bound}}
We conclude this section by proving Theorem \ref{prop:uniform lower bound}. It will be an easy consequence of the following result.

\begin{Prop}\label{prop:equality of A with approximation} Let $d>4$. One has,
\begin{equation}\label{eq:equality of A with approximation}
	\mathbf P^{0\infty,\mathbf{e}_1\infty}[\mathbf C_{\n_1+\n_2}(0)\cap \mathbf C_{\n_1+\n_2}(\mathbf{e}_1)=\emptyset]=\lim_{|x|\wedge |y|\wedge |x-y|\rightarrow \infty}\mathbf P^{0x,\mathbf{e}_1y}_{\beta_c}[\mathbf C_{\n_1+\n_2}(0)\cap \mathbf C_{\n_1+\n_2}(\mathbf{e}_1)=\emptyset].
\end{equation}
\end{Prop}

\begin{proof}[Proof of Theorem \textup{\ref{prop:uniform lower bound}}] Let $c_0:=\tfrac{1}{2}\mathbf P^{0\infty,\mathbf{e}_1\infty}[\mathbf C_{\n_1+\n_2}(0)\cap \mathbf C_{\n_1+\n_2}(\mathbf{e}_1)=\emptyset]$. By \eqref{eq: lower bound on P intro}, one has $c_0>0$. Using Proposition \ref{prop:equality of A with approximation}, there exists $N_0>0$ such that, for every $x,y\in \mathbb Z^d$ with $|x|\wedge|y|\wedge|x-y|\geq N_0$,
\begin{equation}
	\mathbf P^{0x,\mathbf{e}_1y}_{\beta_c}[\mathbf C_{\n_1+\n_2}(0)\cap \mathbf C_{\n_1+\n_2}(\mathbf{e}_1)=\emptyset]\geq c_0.
\end{equation}
This concludes the proof.
\end{proof}
To prove Proposition \ref{prop:equality of A with approximation}, we will need the following lemma. It can be seen as an ``unaveraged'' version of Lemma \ref{lem: localization step}. Recall the definition of $\mathcal A_k$ from \eqref{eq:def Ak}.
\begin{Lem}\label{lem:equality of A with approximation} Let $d>4$. There exist $C,\eta>0$ such that, for every $k\geq 1$, every $x,y\in \mathbb Z^d$ with $x\neq y$ and $|x|,|y|\geq 2k$,
\begin{equation}
	\left|\mathbf P^{0x,\mathbf{e}_1y}_{\beta_c}[\mathbf C_{\n_1+\n_2}(0)\cap \mathbf C_{\n_1+\n_2}(\mathbf{e}_1)=\emptyset]-\mathbf P^{0x,\mathbf{e}_1y}_{\beta_c}[\mathcal A_k]\right|\leq C\left(\frac{1}{|x-y|^{d-4}}\vee\frac{1}{k^\eta}\right).
\end{equation}
\end{Lem}
Proposition \ref{prop:equality of A with approximation} follows almost straightforwardly from Lemma \ref{lem:equality of A with approximation}.
\begin{proof}[Proof of Proposition \textup{\ref{prop:equality of A with approximation}}]
Fix $\beta=\beta_c$ and drop it from the notations. Observe that $(\mathcal A_k)_{k\geq 1}$ is a decreasing sequence of events which satisfies 
\begin{equation}\label{eq:local limit1}
	\bigcap_{k\geq 1}\mathcal A_k=\{\mathbf C_{\n_1+\n_2}(0)\cap \mathbf C_{\n_1+\n_2}(\mathbf{e}_1)=\emptyset\}.
\end{equation}
As a result,
\begin{equation}\label{eq:local limit2}
	\lim_{k\rightarrow \infty}\mathbf P^{0\infty,\mathbf{e}_1\infty}[\mathcal A_k]=\mathbf P^{0\infty,\mathbf{e}_1\infty}[\mathbf C_{\n_1+\n_2}(0)\cap \mathbf C_{\n_1+\n_2}(\mathbf{e}_1)=\emptyset].
\end{equation}
Let $\varepsilon>0$. By \eqref{eq:local limit2}, there exists $k_0>0$ such that, for every $k\geq k_0$, 
\begin{equation}\label{eq:local limit3}
\left|\mathbf P^{0\infty,\mathbf{e}_1\infty}[\mathbf C_{\n_1+\n_2}(0)\cap \mathbf C_{\n_1+\n_2}(\mathbf{e}_1)=\emptyset]-\mathbf P^{0\infty,\mathbf{e}_1\infty}[\mathcal A_{k}]\right|\leq \frac{\varepsilon}{3}.	
\end{equation}
Moreover, by Theorem \ref{thm: IIC rcr}, for every $k\geq 1$, there exists $N_0=N_0(k)>0$ such that, for every $|x|,|y|>N_0$,
\begin{equation}\label{eq:local limit4}
	\left|\mathbf P^{0\infty,\mathbf{e}_1\infty}[\mathcal{A}_{k}]-\mathbf P^{0x,\mathbf{e}_1y}[\mathcal A_{k}]\right|\leq \frac{\varepsilon}{3}.
\end{equation}
Finally, using Lemma \ref{lem:equality of A with approximation}, there exists $\ell=\ell(\varepsilon)\geq k_0$ such that: for every $k\geq \ell$, every $x,y\in \mathbb Z^d$ with $|x|,|y|\geq 2k$ and $|x-y|\geq \ell$,
\begin{equation}\label{eq:local limit5}
		\left|\mathbf P^{0x,\mathbf{e}_1y}[\mathbf C_{\n_1+\n_2}(0)\cap \mathbf C_{\n_1+\n_2}(\mathbf{e}_1)=\emptyset]-\mathbf P^{0x,\mathbf{e}_1y}[\mathcal A_{k}]\right|\leq \frac{\varepsilon}{3}.
\end{equation}
As a consequence of \eqref{eq:local limit3} and \eqref{eq:local limit4} applied to $k=\ell$, together with \eqref{eq:local limit5}, one obtains, for every $x,y\in \mathbb Z^d$ with $|x|,|y|\geq N_0(\ell)$ and $|x-y|\geq \ell$, 
\begin{equation}
	\left|\mathbf P^{0\infty,\mathbf{e}_1\infty}[\mathbf C_{\n_1+\n_2}(0)\cap \mathbf C_{\n_1+\n_2}(\mathbf{e}_1)=\emptyset]-\mathbf P^{0x,\mathbf{e}_1y}[\mathbf C_{\n_1+\n_2}(0)\cap \mathbf C_{\n_1+\n_2}(\mathbf{e}_1)=\emptyset]\right|\leq \varepsilon,
\end{equation}
which concludes the proof.

\end{proof}

We are left with the proof of Lemma \ref{lem:equality of A with approximation}. We rely on the lower bound obtained in \cite{duminil2024new} which, combined with \eqref{eq: IRB}, gives: if $d>4$, there exist $c,C>0$ such that, for every $x\in \mathbb Z^d\setminus\{0\}$,
\begin{equation}\label{eq:up to constant}
	\frac{c}{|x|^{d-2}}\leq \langle \sigma_0\sigma_x\rangle_{\beta_c}\leq \frac{C}{|x|^{d-2}}.
\end{equation}

\begin{proof}[Proof of Lemma \textup{\ref{lem:equality of A with approximation}}]

Fix $\beta=\beta_c$ and drop it from the notations. In the proof, the constant $C$ only depends on $d$ and may change from line to line. Let $k\geq 1$. Recall the definition of $\mathcal A_k$ from \eqref{eq:def Ak}. We will prove the existence of $C,\eta>0$ such that, for every $|x|,|y|\geq 2k$ with $x\neq y$,
\begin{equation}\label{eq:unaveraged0}
	\mathbf P^{0x,\mathbf{e}_1y}[\mathcal B]\leq C\left(\frac{1}{|x-y|^{d-4}}\vee\frac{1}{k^\eta}\right),
\end{equation}
where we recall that $\mathcal B=\mathcal A_k\cap \{\mathbf{C}_{\n_1+\n_2}(0)\cap \mathbf{C}_{\n_1+\n_2}(\mathbf{e}_1)\neq\emptyset\}$.

To prove this result, we import intermediate results and notations from the proof of Lemma \ref{lem: localization step}. Let $\eta>0$ be the constant of Lemma \ref{lem: localization step}. Recall the definitions of $m=k^{\delta}$ and $M=m^A\leq k$ for $A>1,\delta\in(0,1)$. By \eqref{eq:bound k to the eta} we may choose $A,\delta$ so that, for every $k\geq 1$,
\begin{equation}\label{eq:unaveraged1}
	\frac{1}{m^{d-4}}+m^d\left(\frac{k^{(d-1)\delta A}}{k^{d-3}}+\frac{k^{\delta(d-1)}}{k^{A\delta(d-3)}}\right)\leq \frac{C}{k^{\eta}}.
\end{equation}
As in the proof of Lemma \ref{lem: localization step}, we write,
\begin{equation}\label{eq:unaveraged2}
	\mathbf P^{0x,\mathbf{e}_1y}[\mathcal B]\leq \sum_{\substack{u\notin \Lambda_m\\v\sim u}}\mathbf{P}^{0x,\mathbf{e}_1y}[\mathcal B_{u,v}]+Cm^d\mathbf P^{0x,\mathbf{e}_1}[\mathcal G_1\cup \mathcal G_2].
\end{equation}

$\bullet$ Using \eqref{eq: bound u not in lambda m}, we find that
\begin{equation}\label{eq:unaveraged3}
	\sum_{\substack{u\notin \Lambda_m\\v\sim u}}\mathbf{P}^{0x,\mathbf{e}_1y}[\mathcal B_{u,v}]\leq \sum_{\substack{u\notin \Lambda_m\\v\sim u}}\frac{\langle \sigma_{\mathbf{e}_1}\sigma_v\rangle\langle \sigma_v\sigma_y\rangle}{\langle \sigma_{\mathbf{e}_1}\sigma_y\rangle}\frac{\langle \sigma_{0}\sigma_u\rangle\langle \sigma_u\sigma_x\rangle}{\langle \sigma_{0}\sigma_x\rangle}\leq C\sum_{u\notin \Lambda_m}\frac{\langle \sigma_{0}\sigma_u\rangle\langle \sigma_u\sigma_y\rangle}{\langle \sigma_{0}\sigma_y\rangle}\frac{\langle \sigma_{0}\sigma_u\rangle\langle \sigma_u\sigma_x\rangle}{\langle \sigma_{0}\sigma_x\rangle},
\end{equation}
where we used \eqref{eq:up to constant} in the second inequality.
A convolution estimate presented in Appendix \ref{appendix:convolution} (which is licit thanks to \eqref{eq:up to constant}) yields
\begin{equation}\label{eq:unaveraged4}
	\sum_{\substack{u\notin \Lambda_m\\v\sim u}}\mathbf{P}^{0x,\mathbf{e}_1y}[\mathcal B_{u,v}]\leq C\left(\frac{1}{|x-y|^{d-4}}\vee \frac{1}{m^{d-4}}\right).
\end{equation}

$\bullet$ Recall that $\mathcal G_1\subset \mathcal G_1(1)\cup \mathcal G_1(2)$. By \eqref{eq:bound g1(1)}, one has
\begin{align}
	\mathbf P^{0x,\mathbf{e}_1y}[\mathcal G_1(1)]=\mathbf P^{0x}[\mathcal G_1(1)]\leq \sum_{\substack{a\in \partial \Lambda_k\\b \in \partial \Lambda_M}}\frac{\langle \sigma_{0}\sigma_a\rangle\langle \sigma_a\sigma_b\rangle\langle \sigma_b\sigma_{x}\rangle}{\langle \sigma_0\sigma_x\rangle}&\leq C\sum_{\substack{a\in \partial \Lambda_k\\b \in \partial \Lambda_M}}\langle \sigma_{0}\sigma_a\rangle\langle \sigma_a\sigma_b\rangle\langle \sigma_b\sigma_{x}\rangle \notag
	\\&\leq C\frac{k^{(d-1)\delta A}}{k^{d-3}},
\end{align}
where in the second inequality, we used \eqref{eq:up to constant} to argue that $\langle \sigma_b\sigma_x\rangle\leq C\langle \sigma_0\sigma_x\rangle$, and where we used \eqref{eq: bound F_1} in the third inequality. A similar bound holds for $\mathbf P^{0x,\mathbf{e}_1y}[\mathcal G_1(2)]$, so that
\begin{equation}
	\mathbf P^{0x,\mathbf{e}_1y}[\mathcal G_1]\leq C\frac{k^{(d-1)\delta A}}{k^{d-3}}.\label{eq:unaveraged5}
\end{equation}

$\bullet$ Finally, using \eqref{eq: proof localisation f2 (1°}, Claim \ref{claim2}, and \eqref{eq: bound f2}, we get
\begin{equation}\label{eq:unaveraged6}
	\mathbf P^{0x,\mathbf{e}_1y}[\mathcal G_2]\leq C\frac{k^{\delta(d-1)}}{k^{A\delta(d-3)}}.
\end{equation}
Plugging \eqref{eq:unaveraged4}, \eqref{eq:unaveraged5}, and \eqref{eq:unaveraged6} in \eqref{eq:unaveraged2} yields
\begin{equation}\label{eq:unaveraged7}
	\mathbf P^{0x,\mathbf{e}_1y}[\mathcal B]\leq C\left(\frac{1}{|x-y|^{d-4}}\vee \frac{1}{m^{d-4}}\right)+Cm^d\left(\frac{k^{(d-1)\delta A}}{k^{d-3}}+\frac{k^{\delta(d-1)}}{k^{A\delta(d-3)}}\right).
\end{equation}
Combining \eqref{eq:unaveraged7} with \eqref{eq:unaveraged0} concludes the proof.
\end{proof}

\section{Proof of the mixing property}\label{section: mixing d=3}
 
 In this section, we prove Theorem \ref{thm: mixing for currents}. We follow the strategy of \cite{AizenmanDuminilTriviality2021} {\color{black} and refer to this paper for more details}. The argument only requires a modification in one of the steps (namely  \cite[Lemma~6.7]{AizenmanDuminilTriviality2021}), but we provide a full proof for sake of completeness. One essential ingredient to the proof will be the following result.
 \begin{Thm}[Absence of infinite cluster at $\beta_c$, \cite{AizenmanDuminilSidoraviciusContinuityIsing2015}]\label{thm: absence of infinite cluster at criticality} Let $d\geq 3$. Then,
 \begin{equation}
 	\phi_{\beta_c}[0\connect{}\infty]=0.
 \end{equation}
 \end{Thm}
We begin by recalling a fundamental result: the existence of \emph{regular} scales.

\begin{Def}[Regular scales] Fix $c,C>0$. An annular region \textup{Ann}$(n/2,8n)$ is said to be \emph{$(c,C)$-regular} if the following properties hold:
\begin{enumerate}
    \item[$(\mathbf{P1})$] for every $x,y\in \textup{Ann}(n/2,8n)$, $\langle \sigma_0\sigma_y\rangle_\beta\leq C\langle \sigma_0\sigma_x\rangle_\beta$,
    \item[$(\mathbf{P2})$] for every $x,y\in \textup{Ann}(n/2,8n)$, $|\langle \sigma_0\sigma_x\rangle_\beta-\langle \sigma_0\sigma_y\rangle_\beta|\leq \frac{C|x-y|}{|x|}\langle \sigma_0\sigma_x\rangle_\beta$,
    \item[$(\mathbf{P3})$] $\chi_{2n}(\beta)-\chi_n(\beta)\geq c\chi_n(\beta)$,
    \item[$(\mathbf{P4})$] for every $x\in \Lambda_n$ and $y\notin \Lambda_{Cn}$, $\langle \sigma_0\sigma_y\rangle_\beta\leq \frac{1}{2}\langle \sigma_0\sigma_x\rangle_\beta$.
    \end{enumerate}
A scale $k$ is said to be \emph{regular} if $n=2^k$ is such that $\textup{Ann}(n/2,8n)$ is $(c,C)$-regular, a vertex $x \in \mathbb Z^d$ is said to be \emph{in a regular scale} if it belongs to an annulus $\textup{Ann}(n,2n)$ with $n=2^k$ and $k$ a regular scale.
\end{Def}
{\color{black}
Recall that the correlation length $\xi(\beta)$ was defined in \eqref{eq:correlation length}. It diverges as $\beta$ approaches $\beta_c$ (see \cite[Theorem~4.4]{SimonInequalityIsing1980}).
}

\begin{Thm}[Existence of regular scales, {\cite[Theorem~5.12]{AizenmanDuminilTriviality2021}}]\label{thm: existence reg scales} Let $d\geq 3$. Let $\gamma>2$. There exist $\bfc,\bfC>0$ such that, for every $\beta\leq\beta_c$, and every $1\leq n^\gamma\leq N\leq \xi(\beta)$, there are at least $\bfc\log_2\left(\frac{N}{n}\right)$ $(\bfc,\bfC)$-regular scales between $n$ and $N$.
	
\end{Thm}

We are now equipped to prove Theorem \ref{thm: mixing for currents}. Below, $\bfc,\bfC$ are given by Theorem \ref{thm: existence reg scales}.

Let $d\geq 3$ and $\varepsilon>0$. Fix $n\geq 1$, $s\geq 1$, and $t$ satisfying $1\leq t \leq s$. {\color{black}Let $N\geq n$ to be taken large enough. Introduce integers $m,M$ such that $n\leq m\leq M\leq N$.

 Let $\beta\leq \beta_c$. In the proof, we will need many regular scales to ensure that the probabilities of interest ``mix'' with a sufficiently small error. As seen in the statement of Theorem \ref{thm: existence reg scales}, for a given $\beta$ we may ensure the existence of at most $\bfc \log_2(\xi(\beta))$ many such scales. This explains the necessity to take $\beta$ sufficiently close to $\beta_c$ below.}
 
For $\mathbf{x}=(x_1,\ldots,x_t)$ and $\mathbf{y}=(y_1,\ldots,y_t)$, we define
\begin{equation}
    \mathbf{P}^{\mathbf{xy}}_\beta:=\mathbf{P}_\beta^{x_1y_1,\ldots,x_ty_t,\emptyset,\ldots,\emptyset},\qquad \mathbf{P}_\beta^{\mathbf{xy},\emptyset}:=\mathbf{P}_\beta^{\mathbf{xy}}\otimes \mathbf{P}^{\emptyset,\ldots,\emptyset}_\beta,
\end{equation}
where $\mathbf{P}_\beta^{\emptyset,\ldots,\emptyset}$ is the law of a sum of $s$ independent sourceless currents that we denote by $(\n_1',\ldots,\n_s')$. We also let $\mathbf E_\beta^{\mathbf{xy}}$ and $\mathbf E_\beta^{\mathbf{xy},\emptyset}$ be the expectations with respect to these measures. 

{\color{black}We will need the following technical definition which can be compared to the property (\textbf{P2}) of regular scales}. If $p\geq 1$, define for $y\notin \Lambda_{2dp}$,
\begin{equation}
    \mathbb A_y(p):=\left\lbrace u \in \text{Ann}(p,2p) \: : \: \forall x \in \Lambda_{p/d}, \: \langle \sigma_x\sigma_y\rangle_\beta\leq \left(1+\bfC\frac{|x-u|}{|y-u|}\right)\langle \sigma_u\sigma_y\rangle_\beta\right\rbrace.
\end{equation}
Note that if $y$ is in a regular scale, then $\mathbb A_y(p)=\text{Ann}(p,2p)$. Moreover, the MMS inequalities ensure that this set is not empty, see \cite[Remark~6.5]{AizenmanDuminilTriviality2021}.

{\color{black}Let $\mathfrak{K}$ be the set of regular scales $k$ between $m$ and $M/2$ constructed as follows:
\begin{enumerate}
	\item[$(i)$] Let $k_1$ be the smallest $(\bfc,\bfC)$-regular scale in $\text{Ann}(m,M/2)$.
	\item[$(ii)$] Let $\ell\geq 1$ and assume that the regular scale $k_\ell$ is constructed. Choose the smallest regular scale $k>k_\ell$ satisfying both that $2^k\leq M/2$ and $2^{k}\geq \bfC 2^{k_\ell}$. The second condition is useful to apply property (\textbf{P4}) of regular scales. If $k$ exists, then set $k_{\ell+1}:=k$. Otherwise, stop the algorithm.
	\item[$(iii)$] Set $\mathfrak K:=\{k_\ell: \ell\geq 1\}$.
\end{enumerate}
}
By Theorem \ref{thm: existence reg scales}, we may choose $\beta$ sufficiently close to $\beta_c$ and $m,M,N$ large enough so that $\mathfrak{K}\neq \emptyset$. Introduce $\mathbf{U}:=\prod_{i=1}^t\mathbf{U}_i$, where
\begin{equation}
    \mathbf{U}_i:=\frac{1}{|\mathfrak{K}|}\sum_{k\in \mathfrak{K}}\frac{1}{A_{x_i,y_i}(2^k)}\sum_{u\in \mathbb{A}_{y_i}(2^k)}\mathds{1}\{u\connect{\n_i+\n_i'\:}x_i\},
\end{equation}
and,
\begin{equation}
    a_{x,y}(u):=\frac{\langle \sigma_x\sigma_u\rangle_\beta\langle\sigma_u\sigma_y\rangle_\beta}{\langle \sigma_x\sigma_y\rangle_\beta},\qquad A_{x,y}(p):=\sum_{u\in \mathbb A_y(p)}a_{x,y}(u).
\end{equation}
Using the switching lemma \eqref{eq: switching lemma}, we have that $\mathbf{E}_\beta^{\mathbf{xy},\emptyset}[\mathbf{U}]=1$.
The following concentration inequality is a consequence of the definition of $\mathbb A_{y_i}(2^k)$ and the properties of regular scales, and is still valid when $d=3$ (see \cite[Proposition~6.6]{AizenmanDuminilTriviality2021}). The assumption on $m/n$ and $N/M$ below is justified by the definition of $\mathbb A_y(p)$.
\begin{Lem}[Concentration of $\mathbf{U}$]\label{lem: concentration of U} Assume that $x_i\in \Lambda_n$ and $y_i\notin \Lambda_N$ for all $1\leq i \leq t$. Additionally, assume that $m\geq 2dn$ and $N\geq 2dM$. Then, there exists $C=C(d,s)>0$ such that,
\begin{equation}
    \mathbf{E}_\beta^{\mathbf{xy},\emptyset}[(\mathbf{U}-1)^2]\leq \frac{C^2}{|\mathfrak{K}|}.
\end{equation}
\end{Lem}
We work under the assumptions of Lemma \ref{lem: concentration of U}. Using the Cauchy--Schwarz inequality together with Lemma \ref{lem: concentration of U}, 
\begin{equation}\label{eq preuve 1}
\left|\mathbf{P}_\beta^{\mathbf{xy}}[E\cap F]-\mathbf E_\beta^{\mathbf{xy},\emptyset}\big[\mathbf{U}\cdot\mathds{1}\{(\n_1,\ldots,\n_s)\in E\cap F\}\big]\right|\leq \Big(\mathbf{E}_\beta^{\mathbf{xy},\emptyset}[(\mathbf{U}-1)^2]\Big)^{\tfrac{1}{2}}\leq \frac{C}{\sqrt{|\mathfrak{K}|}}. 
\end{equation}

We now introduce the proper event which will allow us to ``decouple'' the events $E$ and $F$.

\begin{Def}\label{def: good event mixing}
Let $\mathbf{u}=(u_1,\ldots,u_t)$ with $u_i\in \textup{Ann}(m,M)$ for every $i$. The event $\mathcal G(u_1,\ldots,u_t)=\mathcal G(\mathbf{u})$ is defined as follows: for every $i\leq s$, there exists $\mathbf{k}_i\leq \n_i+\n_i'$ such that $\mathbf{k}_i=0$ on $\Lambda_n$, $\mathbf{k}_i=\n_i+\n_i'$ outside $\Lambda_N$, $\partial \mathbf{k}_i=\lbrace u_i,y_i\rbrace$ for $i\leq t$, and $\partial \mathbf{k}_i=\emptyset$ for $t<i\leq s$. 
\end{Def}
{\color{black}Observe that under the occurrence of $\mathcal G(\mathbf{u})$, one has $u_i\connect{\n_i+\n_i'\:}y_i$ for every $1\leq i \leq t$.}
By the switching principle (see \cite[Lemma~2.1]{AizenmanDuminilTassionWarzelEmergentPlanarity2019}),\begin{multline}
    \mathbf P_\beta^{\mathbf{x}\mathbf{y},\emptyset}\Big[\{(\n_1,\ldots,\n_s)\in E\cap F\}\cap \:\mathcal G(\mathbf{u})\Big]\\=\Big(\prod_{i=1}^ta_{x_i,y_i}(u_i)\Big)\mathbf P_\beta^{\mathbf{xu},\mathbf{uy}}\left[(\n_1,\ldots,\n_s)\in E, (\n_1',\ldots,\n_s')\in F, \:\mathcal G(\mathbf{u})\right].
\end{multline}
The identity,
\begin{multline}
    \mathbf P_\beta^{\mathbf{xu},\mathbf{uy}}[(\n_1,\ldots,\n_s)\in E, (\n_1',\ldots,\n_s')\in F]\\=\mathbf P_\beta^{\mathbf{xu}}[(\n_1,\ldots,\n_s)\in E]\mathbf P_\beta^{\mathbf{uy}}[(\n_1',\ldots,\n_s')\in F]
\end{multline}
motivates us to prove that under $\mathbf{P}_\beta^{\mathbf{xu},\mathbf{uy}}$, the event $\mathcal G(\mathbf{u})$ occurs with high probability. This is the technical step in the proof of \cite{AizenmanDuminilTriviality2021} that requires some work to be extended to the three-dimensional case. Recall that $n\leq m \leq M \leq N$.
\begin{Lem}\label{technical lemma mixing} We keep the assumptions of Lemma \textup{\ref{lem: concentration of U}}. Assume $\tfrac{\beta_c}{2}\leq \beta\leq \beta_c$. For any fixed value $\alpha=m/M$, there exist $m,N>0$ large enough (and which only depend on $\varepsilon,\alpha,n,s,d$) such that, for every $\mathbf{u}$ with $u_i\in \mathbb A_{y_i}(2^{k_i})$ with $k_i\in\mathfrak{K}$ for $1\leq i \leq t$,
\begin{equation}
    \mathbf P_\beta^{\mathbf{x}\mathbf{y},\emptyset}\Big[\{u_i\connect{\n_i+\n_i'\:}y_i ,\: \forall 1\leq i \leq t\}\cap \:{\color{black}\mathcal G(\mathbf{u})^c}\Big]\Big(\prod_{i=1}^ta_{x_i,y_i}(u_i)\Big)^{-1}=\mathbf P_\beta^{\mathbf{xu},\mathbf{uy}}[\mathcal G(\mathbf{u})^c]\leq \frac{\varepsilon}{2}.
\end{equation}
\end{Lem}

\begin{proof} Below, the constants $C_i$ only depend on $d$ and $s$. The equality follows from an application of the switching lemma, we thus focus on the inequality. Write $\mathcal G(\mathbf{u})=\cap_{1\leq i \leq s}G_i$ (where the definition of $G_i$ is implicit). Then, $H_i^c\cap F_i^c\subset G_i$ where,
\begin{equation}
    H_i:=\lbrace \text{Ann}(M,N) \text{ is crossed by a cluster in }\n_i\rbrace,
\end{equation}
and
\begin{equation}\label{eq: def F_i}
    F_i:=\lbrace \text{Ann}(n,m) \text{ is  crossed by a cluster in }\n_i'\rbrace.
\end{equation}
Indeed, if $H_i^c\cap F_i^c$ occurs, we may define $\mathbf{k}_i$ as the sum of the restriction of $\n_i$ to the clusters intersecting $\Lambda_N^c$ and the restriction of $\n_i'$ to the clusters intersecting $\Lambda_m^c$. Introduce an intermediate scale $n\leq r \leq m$. Using a union bound, 
\begin{equation}\label{eq: prooflast0}
	\mathbf P_\beta^{\mathbf{xu},\mathbf{uy}}[\mathcal G(\mathbf{u})^c]\leq \sum_{i=1}^t\Big(\mathbf P_\beta^{x_iu_i}[H_i]+\mathbf P_\beta^{u_iy_i}[F_i]\Big)+(s-t)\Big(\mathbf P_\beta^{\emptyset}[H_s]+\mathbf P_\beta^\emptyset[F_s]\Big).
\end{equation}
We progressively fix $r,m,N$ large enough (recall that $m/M=\alpha$ is fixed).
\paragraph{Bound on $\mathbf P_\beta^{u_i y_i}[F_i]$.} Fix $1\leq i \leq t$. Notice that
\begin{equation}
	\mathbf P_\beta^{u_i y_i}[F_i]\leq \mathbf P_\beta^{u_i y_i}[\Gamma(\n_i') \textup{ crosses }\textup{Ann}(r,m)]+\mathbf P_\beta^{u_i y_i}[\n_i'\setminus \overline\Gamma(\n_i') \textup{ crosses }\textup{Ann}(n,r)],
\end{equation}
see Figure \ref{fig:mixing}.
Now, using the chain rule (Proposition \ref{prop: chain rule}),
\begin{align}
	\mathbf P_\beta^{u_i y_i}[\Gamma(\n_i') \textup{ crosses }\textup{Ann}(r,m)]&\leq \sum_{v \in \partial \Lambda_r}\frac{\langle \sigma_{u_i}\sigma_v\rangle_\beta \langle \sigma_v\sigma_{y_i}\rangle_\beta}{\langle \sigma_{u_i}\sigma_{y_i}\rangle_\beta}\notag
	\\&\leq C_1\frac{r^2}{m}\max_{v\in \partial \Lambda_r}\frac{\langle \sigma_v\sigma_{y_i}\rangle_\beta}{\langle \sigma_{u_i}\sigma_{y_i}\rangle_\beta}\notag
	\\&\leq C_2\frac{r^2}{m}\label{eq: prooflast1},
\end{align}
where we used the fact that $u\in \mathbb A_{y_i}(2^{k_i})$ on the third line.

\begin{figure}
	\begin{center}
		\includegraphics{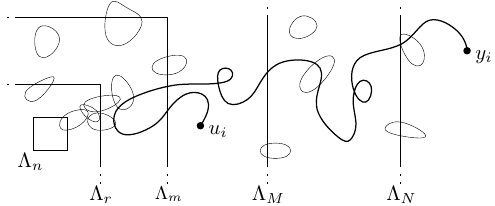}
		\caption{An illustration of a configuration $\n_i'$ contributing to the event $F_i$ introduced in \eqref{eq: def F_i}. The backbone of $\n_i'$ is the bold black line. Here, $\Gamma(\n_i')$ does not cross $\textup{Ann}(r,m)$ but $\n_i'$ crosses $\textup{Ann}(n,m)$ thanks to a string of loops crossing $\textup{Ann}(n,r)$.}
		\label{fig:mixing}
	\end{center}
\end{figure}

Then, arguing as in \eqref{eq: proof localisation f2 (2)} (see also the proof of\cite[Lemma~6.3]{PanisTriviality2023}),
\begin{equation}
	\mathbf P_\beta^{u_i y_i}[\n_i'\setminus \overline\Gamma(\n_i') \textup{ crosses }\textup{Ann}(n,r)]=\sum_{\gamma: u_i\rightarrow y_i}\mathbf P_\beta^{u_iy_i}[\Gamma(\n_i')=\gamma]\mathbf P^\emptyset_{\overline \gamma^c,\beta}[\partial \Lambda_n \connect{\n_i'\:}\partial\Lambda_r].
\end{equation}
However, by Proposition \ref{prop: coupling fk rcr}, one has 
\begin{align}
\mathbf P^\emptyset_{\overline \gamma^c,\beta}[\partial \Lambda_n \connect{\n_i'\:}\partial\Lambda_r]&\leq \phi_{\overline \gamma^c,\beta_c}^0[\partial \Lambda_n \connect{}\partial\Lambda_r]
\\&\leq \phi_{\beta_c}[\partial \Lambda_n \connect{}\partial\Lambda_{r}]\label{eq: prooflast2},
\end{align}
where {\color{black}we used the fact if $\Lambda\subset \mathbb Z^d$, then $\phi_{\beta_c}$ stochastically dominates $\phi^0_{\Lambda,\beta_c}$ (this is a consequence of the domain Markov property as seen \cite[Proposition~1.8]{DuminilLecturesOnIsingandPottsModels2019})}. Using Theorem \ref{thm: absence of infinite cluster at criticality}, we find that for $r$ large enough, one has 
\begin{equation}
\phi_{\beta_c}[\partial \Lambda_n \connect{}\partial\Lambda_{r}]\leq \frac{\varepsilon}{8s}.
\end{equation}
Combining this observation and \eqref{eq: prooflast1}, we obtain that for $m$ large enough,
\begin{equation}\label{eq: prooflast3}
	\mathbf P_\beta^{u_iy_i}[F_i]\leq \frac{\varepsilon}{4s}.
\end{equation}

\paragraph{Bound on $\mathbf P_\beta^\emptyset[F_s]$.} Using Proposition \ref{prop: coupling fk rcr}, 
\begin{equation}
	\mathbf P_\beta^\emptyset[F_s]\leq \phi_{\beta_c}[\partial \Lambda_n\connect{}\partial\Lambda_m].
\end{equation}
Using Theorem \ref{thm: absence of infinite cluster at criticality} and choosing $m$ large enough (and such that \eqref{eq: prooflast3} holds) therefore yields
\begin{equation}\label{eq: prooflast4}
	\mathbf P_\beta^\emptyset[F_s]\leq \frac{\varepsilon}{4s}
\end{equation}

\paragraph{Bound on $\mathbf P_\beta^{x_iu_i}[H_i]$.} Fix $1\leq i \leq t$. Fix $r,m$ such that the above bounds holds. Notice that this fixes the value of $M$ by assumption. Still using Proposition \ref{prop: coupling fk rcr}, one has
\begin{equation}
	\mathbf P_\beta^{x_i u_i}[H_i]{\color{black}\leq}\phi_{\beta}[H_i\:|\:x_i\connect{}u_i]\leq \max_{w,v\in \Lambda_{m}} \phi_{\beta_c/2}[w\connect{}v]^{-1} \cdot \phi_{\beta_c}[\partial \Lambda_M \connect{}\partial\Lambda_N],
\end{equation}
which can be made smaller than $\tfrac{\varepsilon}{4s}$ by using Theorem \ref{thm: absence of infinite cluster at criticality} and choosing $N>0$ large enough.

\paragraph{Bound on $\mathbf P_\beta^{\emptyset}[H_s]$.} We notice that
\begin{equation}
 \mathbf P_\beta^{\emptyset}[H_s]\leq\phi_{\beta_c}[\partial \Lambda_M\connect{} \partial \Lambda_N],
\end{equation}
which we bound similarly by $\tfrac{\varepsilon}{4s}$ for $N$ large enough.

Plugging the above bounds in \eqref{eq: prooflast0}, we obtain
\begin{equation}
	\mathbf P_\beta^{\mathbf{xu},\mathbf{uy}}[\mathcal G(\mathbf{u})^c]\leq \sum_{i=1}^t\Big(\frac{\varepsilon}{4s}+\frac{\varepsilon}{4s}\Big)+(s-t)\Big(\frac{\varepsilon}{4s}+\frac{\varepsilon}{4s}\Big)\leq \frac{\varepsilon}{2},
\end{equation}
which concludes the proof.
\end{proof}

We can now conclude.

\begin{proof}[Proof of Theorem~\textup{\ref{thm: mixing for currents}}]
Introduce the weights $\delta(\mathbf{u},\mathbf{x},\mathbf{y})$ defined by
\begin{equation}
    \delta(\mathbf{u},\mathbf{x},\mathbf{y}):=\mathds{1}\{\exists (k_1,\ldots,k_t)\in \mathfrak{K}^t,\: \mathbf{u}\in \mathbb A_{y_1}(2^{k_1})\times\ldots\times \mathbb A_{y_t}(2^{k_t})\}\prod_{i=1}^t\frac{a_{x_i,y_i}(u_i)}{|\mathfrak{K}|A_{x_i,y_i}(2^{k_i})}.
\end{equation}
Notice that
\begin{equation}
    \sum_{ \substack{(k_1,\ldots,k_t)\in \mathfrak{K}^t\\\mathbf{u}\in \mathbb A_{y_1}(2^{k_1})\times\ldots\times \mathbb A_{y_t}(2^{k_t})}}\delta(\mathbf{u},\mathbf{x},\mathbf{y})=1.
\end{equation}
We abbreviate the above sum as $\sum_{\mathbf{u}}$. Assume that $m/M=\alpha$. Let $\alpha>0$ to be chosen small enough, and let $m=m(\alpha),N=N(\alpha)$ be given by Lemma \ref{technical lemma mixing}.

Let $\mathbf{x}=(x_1,\ldots,x_t)\in (\Lambda_n)^t$ and $\mathbf{y}=(y_1,\ldots,y_t)\in (\Lambda_N^c)^t$.
Equation (\ref{eq preuve 1}) together with Lemma \ref{technical lemma mixing} yield
\begin{equation}
    \Big|\mathbf{P}_\beta^{\mathbf{xy}}[E\cap F]-\sum_{\mathbf{u}}\delta(\mathbf{u},\mathbf{x},\mathbf{y})\mathbf{P}_\beta^{\mathbf{xu}}[E]\mathbf{P}_\beta^{\mathbf{uy}}[F]\Big|\\\leq \frac{C}{\sqrt{|{\color{black}\mathfrak{K}}|}}+\frac{\varepsilon}{2}.
\end{equation}
We now use Theorem \ref{thm: existence reg scales} to argue that if $\beta$ is sufficiently close to $\beta_c$ and if $M/m$ is large enough (which affects how large $N$ is), then $|\mathfrak K|\geq {\color{black}(2C)^2}\varepsilon^{-2}$. This gives for every $\mathbf{y}\in (\Lambda_N^c)^t$,
\begin{equation}\label{eq preuve mixing 1}
	\Big|\mathbf{P}_\beta^{\mathbf{xy}}[E\cap F]-\sum_{\mathbf{u}}\delta(\mathbf{u},\mathbf{x},\mathbf{y})\mathbf{P}_\beta^{\mathbf{xu}}[E]\mathbf{P}_\beta^{\mathbf{uy}}[F]\Big|\\\leq \varepsilon.
\end{equation}
\paragraph{Proof of \eqref{eq 2 mixing}.} We begin by proving (\ref{eq 2 mixing}) when $y_i,y_i'$ are in regular scales (but not necessarily the same ones). Applying the above inequality once for $\mathbf{y}$ and once for $\mathbf{y'}$ with the event $E$ and $F=\Omega_{\mathbb{Z}^d}$,
\begin{equation}\label{eq: preuve mixing }
    \Big|\mathbf{P}_\beta^{\mathbf{xy}}[E]-\mathbf{P}_\beta^{\mathbf{xy'}}[E]\Big|
    \leq 
    \Big|\sum_{\mathbf{u}}(\delta(\mathbf{u},\mathbf{x},\mathbf{y})-\delta(\mathbf{u},\mathbf{x},\mathbf{y'}))\mathbf{P}_\beta^{\mathbf{xu}}[E]\Big|+2\varepsilon.
\end{equation}
Since all the $y_i,y_i'$ are in regular scales, one has $\mathbb{A}_{y_i}(2^{k_i})=\mathbb{A}_{y_i'}(2^{k_i})=\text{Ann}(2^{k_i},2^{k_i+1}).$ Moreover, using Property $\mathbf{(P2)}$ of regular scales, there exists $C_1=C_1(s,d)>0$ such that
\begin{equation}
    \left|\delta(\mathbf{u},\mathbf{x},\mathbf{y})-\delta(\mathbf{u},\mathbf{x},\mathbf{y'})\right|
    \leq C_1  \Big(\frac{M}{N}\Big)\delta(\mathbf{u},\mathbf{x},\mathbf{y})
    \leq \varepsilon\delta(\mathbf{u},\mathbf{x},\mathbf{y}),
\end{equation}
if $N$ is large enough. Indeed, in this configuration $\delta(\mathbf{u},\mathbf{x},\mathbf{y})$ and $\delta(\mathbf{u},\mathbf{x},\mathbf{y}')$ are both close to 
\begin{equation}
    \prod_{i\leq t}\frac{\langle \sigma_{x_i}\sigma_{u_i}\rangle}{|\mathfrak{K}|\sum_{v_i\in \textup{Ann}(2^{k_i},2^{k_i+1})}\langle \sigma_{x_i}\sigma_{v_i}\rangle}.
\end{equation}
This gives that for every $\mathbf{y},\mathbf{y}'\in (\Lambda_N^c)^t$ having all coordinates in regular scales,
\begin{equation}\label{eq: prooflast5}
	\Big|\mathbf{P}_\beta^{\mathbf{xy}}[E]-\mathbf{P}_\beta^{\mathbf{xy'}}[E]\Big|
    \leq 3\varepsilon.
\end{equation}
Note that we can increase $m,M,N$, while maintaining the value $M/m$, and choose $\beta$ even closer to $\beta_c$, to also obtain (using the exact same argument): for every $\mathbf{z},\mathbf{z}'\in (\Lambda_m^c)^t$ having all coordinates in regular scales,
\begin{equation}\label{eq: prooflast6}
	\Big|\mathbf{P}_\beta^{\mathbf{xz}}[E]-\mathbf{P}_\beta^{\mathbf{xz'}}[E]\Big|
    \leq 3\varepsilon.
\end{equation}
We now consider $\mathbf{z}=(z_1,\ldots,z_t)$ with $z_i\in \text{Ann}(m,M)$ in a regular scale. Also, pick $\mathbf{y}$ on which we do not assume anything. We have,
\begin{eqnarray*}
    \Big|\mathbf{P}_\beta^{\mathbf{xy}}[E]-\mathbf{P}_\beta^{\mathbf{xz}}[E]\Big|&=&\Big|\mathbf{P}_\beta^{\mathbf{xy}}[E]-\sum_{\mathbf{u}}\delta(\mathbf{u},\mathbf{x},\mathbf{y})\mathbf{P}_\beta^{\mathbf{xz}}[E]\Big|
    \\&\leq& 
    \Big|\mathbf{P}_\beta^{\mathbf{xy}}[E]-\sum_{\mathbf{u}}\delta(\mathbf{u},\mathbf{x},\mathbf{y})\mathbf{P}_\beta^{\mathbf{xu}}[E]\Big|+3\varepsilon    \\&\leq& 4\varepsilon,
\end{eqnarray*}
where on the second line we used \eqref{eq: prooflast6} with $(\mathbf{y},\mathbf{y}')=(\mathbf{u},\mathbf{z})$, and in the third line we used (\ref{eq preuve mixing 1}) with $F=\Omega_{\mathbb Z^d}$. This gives (\ref{eq 2 mixing}). 

\paragraph{Proof of \eqref{eq 3 mixing}.} The same argument works for (\ref{eq 3 mixing}) for every $\mathbf{x},\mathbf{x'},\mathbf{y}$, noticing that for every  regular $\mathbf{u}$ for which $\delta(\mathbf{u},\mathbf{x},\mathbf{y})\neq 0$, using once again $\mathbf{(P2)}$,
\begin{equation}
    \left|\delta(\mathbf{u},\mathbf{x},\mathbf{y})-\delta(\mathbf{u},\mathbf{x'},\mathbf{y})\right|
    \leq C_2  \left(\frac{n}{m}\right)\delta(\mathbf{u},\mathbf{x},\mathbf{y})\leq \varepsilon\delta(\mathbf{u},\mathbf{x},\mathbf{y})
\end{equation}
if we take $m$ sufficiently large.

\paragraph{Proof of \eqref{eq 1 mixing}.} To obtain (\ref{eq 1 mixing}) we repeat the same line of reasoning. We start by applying (\ref{eq preuve mixing 1}). Then, we replace each $\mathbf{P}_\beta^{\mathbf{xu}}[E]$ by $\mathbf{P}_\beta^{\mathbf{xy}}[E]$ using \eqref{eq: prooflast6}. Similarly, we replace $\mathbf P_\beta^{\mathbf{uy}}[F]$ by $\mathbf{P}_\beta^{\mathbf{xy}}[F]$. The proof follows readily.
\end{proof}

\section{Proof of Proposition~\ref{prop: prop of rc IIC}}\label{section: proof of prop iic rc}
In this last section, we use the techniques introduced in Section  \ref{section: mixing d=3} to prove Proposition \ref{prop: prop of rc IIC}.

\begin{proof}[Proof of Proposition \textup{\ref{prop: prop of rc IIC}}] 
We begin with $(i)$. The argument follows from the fact that the IIC of the measure of interest can be described by the union of the backbone and a collection of finite loops. If the IIC has two ends, then there must be an infinite path which does not use any edges of the backbone. {\color{black}Fix $x,y\in \mathbb Z^d$ and $n\geq 1$}. Reasoning again as in \eqref{eq: bound C(s,t)} {\color{black}and using \eqref{eq: IRB} gives $C_1=C_1(d)>0$ such that}  
\begin{equation}
	\mathbf P^{0x,\emptyset}_{\beta_c}[y \connect{\n_1+\n_2\setminus\overline\Gamma(\n_1)\:}\partial \Lambda_n]\leq  \sum_{u\in \partial \Lambda_n}\langle \sigma_u\sigma_y\rangle_{\beta_c}^2\leq \frac{C_1}{n^{d-3}}.
\end{equation}
Hence, by Theorem \ref{thm: IIC rcr},
\begin{equation}
	\mathbf P^{0\infty,\emptyset}[y \connect{\n_1+\n_2\setminus\overline\Gamma(\n_1)\:}\partial \Lambda_n]\leq \frac{C_1}{n^{d-3}},
\end{equation}
which immediately yields,
\begin{equation}
	\mathbf P^{0\infty,\emptyset}[y \connect{\n_1+\n_2\setminus\overline\Gamma(\n_1)\:}\infty]=0
\end{equation}
and the result.

We turn to $(ii)$.  Notice that by the switching lemma, if $x,y\in \mathbb Z^d$,
\begin{equation}
	\mathbf P_{\beta_c}^{0x,\emptyset}[0\connect{}y]=\frac{\langle \sigma_0\sigma_y\rangle_{\beta_c}\langle \sigma_y\sigma_x\rangle_{\beta_c}}{\langle \sigma_0\sigma_x\rangle_{\beta_c}}.
\end{equation}
The gradient estimate \cite[Proposition~5.9]{AizenmanDuminilTriviality2021} gives that
\begin{equation}
	\lim_{|x|\rightarrow \infty}\frac{\langle \sigma_y\sigma_x\rangle_{\beta_c}}{\langle \sigma_0\sigma_x\rangle_{\beta_c}}=1,
\end{equation}
which seems to imply \eqref{eq: connectivity in iic rcr} by Theorem \ref{thm: IIC rcr}. However, the event $\{0\connect{}y\}$ is not local\footnote{For each fixed $x$, this event can be approximated by local events using the fact that the measure $\mathbf P_{\beta_c}^{0x,\emptyset}$ does not percolate, see\cite{AizenmanDuminilSidoraviciusContinuityIsing2015}. Here we need a version of that statement which is uniform in $x$. This is why the qualitative result of \cite{AizenmanDuminilSidoraviciusContinuityIsing2015} is not enough when $d=3$.}, so we need to work a little to conclude. Let $2|y|\leq m\leq M$ to be taken large enough. We follow the argument used in the proof of Lemma \ref{technical lemma mixing} and notice that
\begin{equation}
	\mathbf P_{\beta_c}^{0x,\emptyset}[\{0\connect{\Lambda_M}y\}^c\cap \{0\connect{}y\}]\leq \mathbf P_{\beta_c}^{0x}[\mathcal H_1]+\mathbf P_{\beta_c}^{0x,\emptyset}[\mathcal H_2],
\end{equation}
where $\mathcal H_1$ and $\mathcal H_2$ are two events defined as follows:
\begin{enumerate}
	\item[-] $\mathcal H_1$: the backbone $\Gamma(\n_1)$ does two successive crossings of $\textup{Ann}(m,M)$,
	\item[-] $\mathcal H_2$: the event $\mathcal H_1$ does not occur and there exists $s\in \Lambda_m$ such that $y\connect{\n_1+\n_2\setminus\overline\Gamma(\n_1)\:}s$.
\end{enumerate}
Using the chain rule for backbones and \eqref{eq: IRB},
\begin{equation}
	\mathbf P_{\beta_c}^{0x}[\mathcal H_1]\leq \sum_{\substack{u\in \partial \Lambda_m\\v\in \partial \Lambda_M}}\frac{\langle \sigma_0\sigma_v\rangle_{\beta_c}\langle \sigma_v\sigma_u\rangle_{\beta_c}\langle \sigma_u\sigma_x\rangle_{\beta_c}}{\langle \sigma_0\sigma_x\rangle_{\beta_c}}\leq \max_{u\in \Lambda_m}\frac{\langle \sigma_u\sigma_x\rangle_{\beta_c}}{\langle \sigma_0\sigma_x\rangle_{\beta_c}} \cdot C_2\frac{m^{d-1}}{M^{d-3}},
\end{equation}
where $C_2=C_2(d)>0$.
Moroever, a similar reasoning as in \eqref{eq: bound C(s,t)} yields the existence of $C_3=C_3(d)>0$ such that,
\begin{equation}
	\mathbf P_{\beta_c}^{0x,\emptyset}[\mathcal H_2]\leq \sum_{u\in \partial \Lambda_m}\langle \sigma_u\sigma_y\rangle_{\beta_c}^2\leq \frac{C_3}{m^{d-3}},
\end{equation}
where we used \eqref{eq: IRB} in the last inequality. Taking $m=M^{1/4}$, {\color{black}and using the gradient estimate \cite[Proposition~5.9]{AizenmanDuminilTriviality2021} one more time to argue that
\begin{equation}
	\lim_{|x|\rightarrow \infty}\max_{u\in \Lambda_m}\frac{\langle \sigma_u\sigma_x\rangle_{\beta_c}}{\langle \sigma_0\sigma_x\rangle_{\beta_c}} =1,
\end{equation}
yields}
\begin{equation}
	\limsup_{M\rightarrow \infty}\limsup_{|x|\rightarrow \infty}\mathbf P_{\beta_c}^{0x,\emptyset}[\{0\connect{\Lambda_M}y\}^c]=0.
\end{equation}
The proof of \eqref{eq: connectivity in iic rcr} follows readily. 

Finally, \eqref{eq: connectivity iic rcr d>3} follows from \eqref{eq: IRB} and \cite[Theorem~1.4]{duminil2024new}.

 \end{proof}

\begin{Acknowledgements} We thank Hugo Duminil-Copin for inspiring discussions and for constant support. We thank Gordon slade for useful comments leading to the discussion below Theorem \ref{thm: value of the constant} and to Proposition \ref{prop: prop of rc IIC}. We also thank Lucas D'Alimonte, Trishen S. Gunaratnam, Yucheng Liu, Mathilde Pacholski, Alexis Prévost, and two referees for useful comments. This project has received funding from the Swiss National Science Foundation, the NCCR SwissMAP, and the European Research Council (ERC) under the European Union’s Horizon 2020 research and innovation programme (grant agreement No. 757296). \end{Acknowledgements}

\appendix
{\color{black}
\section{A convolution estimate}\label{appendix:convolution}

\begin{Prop} Let $d>4$. Let $f:\mathbb Z^d\rightarrow\mathbb R^*$. Assume that there exist $c_1,C_1>0$ such that, for every $x\in \mathbb Z^d\setminus\{0\}$,
\begin{equation}\label{eq:assumption f}
	\frac{c_1}{|x|^{d-2}}\leq f(x)\leq \frac{C_1}{|x|^{d-2}}.
\end{equation}
Then, there exist $C>0$ such that, for every $m\geq 1$, for every $x,y\in \mathbb Z^d$ with $x\neq y$ and $|x|,|y|\geq 2m$,
\begin{equation}
	\sum_{u\notin \Lambda_m}f(u)f(x-u)f(u)f(y-u)\leq C\left(\frac{1}{|x-y|^{d-4}}\vee \frac{1}{m^{d-4}}\right)f(x)f(y).
\end{equation}
\end{Prop}
\begin{proof} In this proof, the constant $C$ only depends on $d$ and may change from line to line. Let $m\geq 1$ and $x,y\in \mathbb Z^d$ be as above. Without loss of generality, we assume that $|x|\geq |y|$. We split the sum we want to estimate into three contributions: $u\in \Lambda_{|x|/2}(x)$, $u\in \Lambda_{|y|/2}(y)\setminus \Lambda_{|x|/2}(x)$, and $u\in \Lambda_m^c\setminus \Big(\Lambda_{|x|/2}(x)\cup \Lambda_{|y|/2}(y)\Big)$. We respectively denote them by $(I)$, $(II)$, and $(III)$.
\paragraph{Bound on $(I)$.} If $u\in \Lambda_{|x|/2}(x)$, then by \eqref{eq:assumption f}, one has $f(u)\leq Cf(x)^2\leq Cf(x)f(y)$ (since $|x|\geq |y|$). Moreover, using \eqref{eq:assumption f} again,
\begin{equation}\label{eq:convol1}
	\sum_{u\in \Lambda_{|x|/2}(x)}f(x-u)f(y-u)\leq \sum_{u\in \mathbb Z^d}f(x-u)f(y-u)\leq \frac{C}{|x-y|^{d-4}}.
\end{equation}
As a result,
\begin{equation}
	(I)\leq \frac{C}{|x-y|^{d-4}}f(x)f(y).
\end{equation}
\paragraph{Bound on $(II)$.} If $u\in \Lambda_{|y|/2}(y)\setminus\Lambda_{|x|/2}(x)$, then $f(x-u)\leq Cf(x)$ and $f(u)\leq Cf(y)$. Moreover, still by \eqref{eq:assumption f},
\begin{equation}\label{eq:convol2}
	\sum_{u\in \Lambda_{|y|/2}(y)\setminus\Lambda_{|x|/2}(x)}f(u)f(y-u)\leq \sum_{u\in \mathbb Z^d}f(u)f(y-u)\leq \frac{C}{|y|^{d-4}}\leq \frac{C}{m^{d-4}}.
\end{equation}
\paragraph{Bound on $(III)$.} If $u\notin \Lambda_{|x|/2}(x)\cup \Lambda_{|y|/2}(y)$, then $f(x-u)\leq Cf(x)$ and $f(y-u)\leq Cf(y)$ by \eqref{eq:assumption f}. By \eqref{eq:assumption f}, one has
\begin{equation}\label{eq:convol3}
	\sum_{u\in \Lambda_m^c\setminus\Big(\Lambda_{|x|/2}(x)\cup \Lambda_{|y|/2}(y)\Big)}f(u)^2\leq \frac{C}{m^{d-4}}.
\end{equation}
Combining \eqref{eq:convol1}, \eqref{eq:convol2}, and \eqref{eq:convol3} concludes the proof.
\end{proof}
}

\bibliographystyle{alpha}
\bibliography{biblio}
\end{document}